\documentclass[a4paper, twoside, 10pt]{article}
\usepackage{eurosym, xcolor}
\usepackage{amsmath,amsthm,amsfonts,latexsym,amscd,amssymb, enumerate}
\usepackage{hyperref}
\usepackage{enumitem}
\setlist[enumerate]{itemsep=0pt, topsep=0pt, parsep=0pt}
\numberwithin{equation}{section}
\input xypic
\newtheorem{theorem}{Theorem}[section]
\newtheorem{lemma}{Lemma}[section]

\newtheorem{proposition}{Proposition}[section]
\newtheorem{definition}{Definition}[section]
\newtheorem{example}{Example}[section]
\theoremstyle{remark}
\newtheorem{remark}{Remark}[section]
\date{}
\usepackage{amsmath}
\DeclareMathOperator{\Div}{div}
\DeclareMathOperator{\grad}{grad}
\DeclareMathOperator{\Ric}{Ric}
\title{\textbf{On the geometry of Riemannian warped product  maps}}
\author{Jyoti Yadav, Harmandeep Kaur and Gauree Shanker\thanks{corresponding author, Email: gauree.shanker@cup.edu.in}	 }

\begin{document}
\maketitle
\begin{abstract}
In this paper, we begin by introducing Clairaut Riemannian warped prod-
uct maps and establish the condition under which a regular curve becomes a geodesic. We obtain the conditions for a Riemannian warped product map to be Clairaut Riemannian warped product map followed by Ricci curvature. Further, we study the Ricci soliton structure on a Riemannian warped product manifold using curvature tensor. We examine the Bochner type formulae for Clairaut Riemannian warped product map and construct a supporting example.
Furthermore, we extend the study to introduce and examine some geometric aspects of conformal Riemannian warped product maps. We derive the integral formula for scalar curvature of conformal Riemannian warped product map. Finally, we construct an example for conformal Riemannian warped product map.

  		
	\end{abstract}
	\noindent\textbf{Mathematics Subject Classification}. 53B20, 53C22, 53C42, 53C55\\
	\textbf{Keywords}. Riemannian map, Clairaut Riemannian map, conformal Riemannian map,  Riemannian warped product submersion, Riemannian warped product map.
	
\section{Introduction}
Warped product manifolds, introduced by Bishop and O'Neill \cite{BishopNeill}, as a genera lization of product manifold, provide a framework for constructing new examples of manifolds with negative curvature by defining suitable convex functions on them. Schwarzschild and Robertson-Walker models \cite{Neill} are the prominent examples of warped products. In Riemannian geometry, de Rham theorem states that Riemannian manifold can be locally decomposed into product manifold. Subsequently, Moore \cite{J17} demonstrated the sufficient conditions for an isometric immersion into an Euclidean space to decompose into a product immersion. \\
The definition of warped product immersion is as follows: Let $\phi_i: M_i\rightarrow N_i$ be isometric immersions, where $i= 1,\ldots,k.$ Assume $M_1\times_{f_1} M_2\times_{f_2},\ldots,\times_{f_{k-1}}M_k$ and $N_1\times_{\rho_1} N_2\times_{\rho_2},\ldots,\times \rho_{k-1}N_k$ be warped product manifolds, where $\rho_i:N_i\rightarrow\mathbb{R}^+$ and $ f_i=\rho_i\circ\phi_i:M_i\rightarrow\mathbb{R}^+$ for $i=1,\ldots,k-1,$ are smooth functions. Then, the smooth map $\phi:M_1\times_{f_1} M_2\times_{f_2},\ldots,\times_{f_{k-1}}M_k\rightarrow N_1\times_{\rho_1} N_2\times_{\rho_2},\ldots,\times_{\rho_{k-1}}N_k,$ defined by $\phi(p_1,\ldots,p_k)= (\phi_1(p_1),\ldots,\phi_k(p_k))$ is an isometric immersion. For further study of warped product isometric immersion, we refer \cite{Chen, J9}.\\	
A crucial problem in Riemannian geometry is to construct Riemannian manifolds with positive or non-negative sectional curvature. Numerous examples of Einstein manifolds have been constructed through the use of Riemannian submersions. The concept of Riemannian submersion was introduced by Gray and O'Neill \cite{Gray,J10}. For a more comprehensive study of Riemannian submersions, we refer \cite{J13}.\\
Erken et. al \cite{J7} have studied Riemannian submersions between warped product manifolds. They have calculated fundamental tensor fields with various cases of vector fields. In this context, Erken et al. \cite{ErkenSiddiqui} investigated certain inequalities involving Riemannian warped product submersions with respect to the vertical Casorati curvatures. For further studies on Riemannian warped product submersions, see \cite{J7}.\\ As a generalization of isometric immersions and Riemannian submersions, the concept of a Riemannian map was introduced by Fischer \cite{Fischer}. A notable feature of Riemannian maps is that they satisfy a generalized Eikonal equation, which serves as a conceptual link between physics and differential geometry. For a detailed study of Riemannian maps, we refer \cite{sahin book}. The geometry of Riemannian warped product maps has been further explored by Meena et al. \cite{J14}.\\  
In addition, the theory of conformal Riemannian maps, which generalizes Riemannian maps, was introduced by Şahin. In this framework, the differential of the map is a conformal isometry when restricted to the horizontal distribution. Tojeiro \cite{Tojeiro} have studied the geometry of conformal immersions of warped product manifolds. In this paper, we aim to introduce the concept of conformal Riemannian warped product maps and to explore the geometric properties of warped product manifolds.\\ In 1972, Bishop \cite{Bishop} generalized the idea of Clairaut's theorem in surface theory to higher dimensional surfaces and introduced Clairaut submersion. Later, in 2017,  \c{S}ahin generalized Bishop's idea and introduced Clairaut Riemannian map.\\
The idea of Ricci soliton is originated from Hamilton's work \cite{Hamilton}. It is the generalization of Einstein metric. Perelman used the idea of Ricci soliton to prove the important conjecture, called Poincar\'e conjecture. The study of Ricci soliton is a fascinating field in mathematics and physics. A Riemannian manifold $(M, g)$ admits a Ricci soliton structure if there exist a vector field $\xi,$ called potential vector field such that
\begin{equation}
\frac{1}{2}L_{\xi}g+Ric+\lambda g = 0.
\end{equation} 
The behavior of Ricci soliton depends on the nature of $\lambda.$ If $\lambda$ is positive, zero or negative then Ricci soliton will be expanding,  steady or shrinking, respectively. If $\lambda$ is a non constant function, then it becomes almost Ricci soliton.\\
In this paper, we introduce the concepts of Clairaut Riemannian warped product maps and conformal Riemannian warped product maps, and establish several significant results related to their geometric properties. In section 2, we outline some important definitions and results, required for current research work. In section 3, we determine the conditions for a regular curve to be a geodesic and derive the necessary and sufficient conditions for a Riemannian warped product map to be a Clairaut Riemannian warped product map. We compute the Ricci curvature tensor associated with Clairaut Riemannian warped product maps. Additionally, we explore the Ricci soliton structure on warped product manifolds. we then formulate Bochner-type formulae for these maps. To complement the theoretical findings, we conclude the section with a concrete example. In Section 4, we define the concept of a conformal Riemannian warped product map, investigate the geometric properties of warped product manifolds under such mappings, and provide an illustrative example to demonstrate the theory.
\section{Preliminaries}
Here, we recall some definitions and results which are relevant to this paper.
\subsection{Riemannian map and O'Neill tensor}
Let $\phi_{1}:(M^{m_1}_1, g_{M_1})\rightarrow(N^{n_1}_1, g_{N_1})$ be a smooth map between two Riemannian manifolds, where $0< \text{rank}~ \phi_{1}\leq \min\{m_1, n_1\}.$  For given $p_1\in M_1, \mathcal{V}_{p_1} = \ker\phi_{1*p_1}$ denotes the vertical distribution, where $\mathcal{V}:p_1 \mapsto\mathcal{V}_{p_1}$ and orthogonal compliment of  $\mathcal{V}_{p_1}$ is $\mathcal{H}_{p_1} = (\ker \phi_{1*p_1})^\perp,$ where $\mathcal{H}:p_1 \mapsto \mathcal{H}_{p_1}.$ Therefore, the tangent space $T_{p_1}M_1$ can be written as 
$$T_{p_1}M_1 =  \ker\phi_{1*p_1}\oplus (\ker\phi_{1*p_1})^\perp.$$
Moreover, range of $\phi_{1*p_1}$ is denoted by $\text{range} ~\phi_{1*p_1}$ and its orthogonal compliment is given by $(\text{range}~ \phi_{1*p_1})^\perp.$ Thus tangent space at $\phi_{1}(p_1) = q_1,$
$T_{q_1}N_1$ is the direct sum of $\text{range}~ \phi_{1*p_1}$ and $(\text{range}~ \phi_{1*p_1})^\perp,$ i.e., $T_{q_1}N_1 = \text{range}~ \phi_{1*p_1}\oplus (\text{range}~ \phi_{1*p_1})^\perp.$
\begin{definition}\cite{sahin book}
A smooth map $\phi_{1}$ is said to be a Riemannian map at $p_1\in M_1$ if the horizontal restriction
$\phi_{1p_1}^{h}:(\ker\phi_{1*p_1})^\perp\rightarrow (\text{range}~\phi_{1*p_1})$ is a linear isometry between the inner product spaces $\big((\ker\phi_{1*p_1})^\perp, g_{M_1}{(p_1)}|_{ (\ker\phi_{1*p_1})^\perp}\big)$ and $\big((\text{range}~\phi_{1*p_1}, g_{N_1}(q_1)|_{(\text{range}~\phi_{1*p})}\big),$ i.e.,
$$g_{N_1}(\phi_{1*}X, \phi_{2*}Y) = g_{M_1}(X, Y).$$
\end{definition}
O'Neill \cite{J10} defined fundamental tensor fields $T_1$
and $A_1$ by
\begin{equation}\label{neill T}
T^1_{X_1}{Y_1} =  \mathcal{H}_1\nabla^M_{\mathcal{V}_1X_1}\mathcal{V}_1Y_1 +\mathcal{V}_1\nabla^{M_1}_{\mathcal{V}_1X_1}\mathcal{H}_1Y_1,\\
\end{equation}
\begin{equation}\label{neill A}
A^1_{X_1}{Y_1} = \mathcal{H}_1\nabla^{M_1}_{\mathcal{H}_1X_1}\mathcal{V}_1Y_1 + \mathcal{V}_1\nabla^{M_1}_{\mathcal{H}_1X_1}\mathcal{H}_1Y_1,\\
\end{equation}
where $X_1, Y_1\in\Gamma(TM_1), \nabla^{M_1}$ is the Levi-Civita connection of $g_{M_1}$ and $\mathcal{V}_1, \mathcal{H}_1$ denote the projections to vertical subbundle and horizontal subbundle,
respectively. For any $X_1\in \Gamma(TM_1)$, $T_{X_1}$ and $A_{X_1}$ are skew-symmetric operators on $(\Gamma(TM_1), g_{M_1})$ reversing the horizontal and the vertical distributions. One can check easily that $T^1$ is vertical, i.e., $T^1_{X_1} = T^1_{\mathcal{V}_1 X_1}$, and $A^1$ is horizontal, i.e., $A^1 = A^1_{\mathcal{H}_1X_1}$. We observe that the tensor field $T^1$ is symmetric for vertical vector fields, 
and tensor field $A^1$ is anti-symmetric for horizontal vector fields. 
Using  equation  \eqref{neill T} and \eqref{neill A}, we have the following Lemma.\\
\begin{lemma} \cite{J10}\label{Neill tensor} Let  $X, Y\in\Gamma (\ker\phi_{*})^\perp$ and $V,W \in \Gamma(\ker\phi_{*}).$ Then
\begin{equation*}
\nabla_{V}W = T_{V}W + \hat{\nabla}_{V}W,
\end{equation*}
\begin{equation*}
\nabla_{V}X = \mathcal{H}\nabla_{V}X + T_{V}X,
\end{equation*}
\begin{equation*}
\nabla_{X}V = A_{X}V +\mathcal{V}\nabla_{X}V,
\end{equation*}
\begin{equation*}
\nabla_{X}Y = \mathcal{H}\nabla_{X}Y + A_{X}Y.
\end{equation*}
If $X$ is a basic vector field, then $\mathcal{H}\nabla_{V}X = A_{X}V.$
\end{lemma}
\begin{definition}\cite{sahin book}
Let $\phi_{1}:(M_1, g_{M_1})\rightarrow (N_1, g_{N_1})$ be a smooth map between two Riemannian manifolds. Then we say that $\phi_1$ is a conformal Riemannian
map at $p_1\in M_1$ if $0 < \text{rank}~\phi_{1*}{p_1} \leq \min\{m_1, n_1\}$ and $\phi_{1*}{p_1}$ maps the horizontal space $\mathcal{H}_{p_1} =
(\ker(\phi_{1*}{p_1}))^\perp$ conformally onto $\text{range}~\phi_{1*}{p_1},$ i.e., there exists a number $\lambda^2({p_1}) \neq 0$ such that
$g_{N_1}(\phi_{1*}X_1, \phi_{1*}Y_1) = \lambda^2(p_1)g_{M_1}(X_1, Y_1)$
for $X,Y \in \mathcal{H}(p_1).$ Also $\phi_{1}$ is called conformal Riemannian if $\phi_{1}$ is conformal Riemannian at
each $p_1 \in M_1.$
\end{definition}
\begin{definition}
\cite{J3} Let $(M_1, g_{M_1})$ and $(N_1, g_{N_1})$ be two Riemannian manifolds. A Riemannian map $\phi_{1} : (M_1, g_{M_1})\rightarrow (N_1, g_{N_1})$ is called a Clairaut Riemannian map if there is a function
$r : M_1 \rightarrow \mathbb{R}^+$ such that for every geodesic, making angles $\theta$ with the horizontal subspaces,
$r \sin \theta$ is constant.
\end{definition}
\subsection{Riemannian warped product manifolds}
Let $(M_1, g_{M_1})$ and $(M_2, g_{M_2})$ be two finite dimensional Riemannian manifolds with $\dim M_1 = m_1, \dim M_2 = m_2,$
respectively. Let $f$ be a positive smooth function on $M_1.$ The warped product $M = M_1 \times_f M_2$ of $M_1$ and $M_2$ is the Cartesian product $M_1 \times M_2$ with the metric $g_M = g_{M_1} + f^2g_{M_2}.$ It means the Riemannian metric $g_M$ on $M = (M_1\times_fM_2)$ can be defined for any two vecor fields $X, Y\in M$ in the following manner
$$g_M(X, Y) = g_{M_1}(\pi_{1*}X, \pi_{1*}Y) + f^2(\pi_1(p_1)) g_{M_2}(\pi_{2*}X, \pi_{2*}Y),$$ where $\pi_1 : M_1 \times_f M_2 \rightarrow M_1$ defined by $(x, y) \rightarrow x$ and $\pi_{2} : M_1 \times_f M_2 \rightarrow M_2$ defined by $(x, y) \rightarrow y$ are projections which are Riemannian submersion. Further, one can observe that the fibers $\{x\}\times M_2 = \pi_{1}^{-1}(x)$ and the leaves $M_1\times\{y\} = \pi_{2}^{-1}(y)$ are Riemannian submanifolds of $M_1 \times_f M_2.$ The vectors which are tangent to leaves and fibers are called horizontal and vertical vectors, respectively. If $v \in T_xM_1, x \in M_1$ and $y \in M_2,$ then the lift $\bar{v}$ of $v$ to $(x, y)$ is the unique vector $T_{(x,y)}M$ such that $(\pi_{1*})\bar{v} = v.$ For a vector field, $X \in\Gamma(T{M_1}),$ the lift of $X$ to $M=M_1\times_fM_2$ is the vector field, $\bar{X}$ whose value at $(x, y)$ is the lift of $X_x$ to $(x, y).$ Thus the lift of $X \in\Gamma(T{M_1})$ to $M_1\times_f M_2$ is the unique element of $\Gamma(T(M_1\times M_2))$ that is $\pi_{1}$-related to $X.$ The set of all horizontal lifts and vertical lifts is denoted by
$\mathcal{L}_H(M_1),$  $\mathcal{L}_V(M_2)$, respectively. Thus a vector field $\bar{E}$ of $M_1 \times_fM_2$ can be expressed as
$\bar{E} = \bar{X}+\bar{U}$ with $\bar{X}\in \mathcal{L}_H(M_1)$ and $\bar{U}\in \mathcal{L}_V(M_2).$ Hence, we can write 
$$\pi_{1*}(\mathcal{L}_H(M_1)) = \Gamma(TM_1), \pi_{2*}(\mathcal{L}_V(M_2)) = \Gamma(TM_2),$$ and so $\pi_{1*}(\bar{X}) = X \in\Gamma(TM_1)$ and $\pi_{2*}(\bar{U}) = U \in\Gamma(TM_2).$
Throughout this paper, we take same notation for vector field and its lift to warped product manifold.
\begin{definition}\cite{J14}
Let $\phi_i:M_i\rightarrow N_i$ be the Riemannian maps, where $i=1, 2.$ Then the map
$$\phi(= \phi_1\times\phi_2):M (= M_1\times_fM_2)\rightarrow N ( = N_1\times_\rho N_2)$$ defined as $(\phi_1\times\phi_2)(p_1, p_2) = (\phi_1(p_1), \phi_2(p_2))$ is a Riemannian warped product map.
\end{definition}
\begin{lemma}\label{lemma}\cite{chenbook}
Let $M = M_1 \times_f M_2$ be a warped product manifold. For any $X_1, Y_1\in \mathcal{L}(M_1)$ and $X_2, Y_2 \in \mathcal{L}(M_2),$ we have
\begin{enumerate}
\item $\nabla^M_{X_1} Y_1$ is the lift of $\nabla^{M_1}_{X_1} Y_1,$
\item  $\nabla ^M_{X_1} X_2 = \nabla^M_{X_2}X_1 = \frac{X_1(f)}{f}X_2,$
\item  nor $(\nabla^M_{X_2}Y_2) = -g_M(X_2, Y_2)(\nabla^M\ln f),$
\item  tan $(\nabla^M_{X_2}Y_2)$ is the lift of $\nabla^{M_2}_{X_2}Y_2.$	
\end{enumerate}
Here $\nabla^M,\nabla^{M_1}$ and $\nabla^{M_2}$ denote Riemannian connections on $M, M_1$ and $M_2,$ respectively.
\end{lemma}
\begin{proposition}\cite{J15}
Let $\phi_i: M_i\rightarrow N_i,$  i = 1, 2, be smooth functions. Then
$$(\phi_1\times\phi_2)_*w = (\phi_{1*}u, \phi_{2*}v),$$
where $w = (u, v) \in T_{p}(M_1 \times M_2)$ and $p = (p_1, p_2).$
\end{proposition}
\begin{proposition}\label{Ismphm}\cite{J15}
Let $\pi_{1}$ and $\pi_{2}$ be projections from $M_1\times_fM_2$ onto $M_1$ and $M_2,$ respectively. Then $\lambda: T_{(p_1,p_2)}(M_1 \times M_2)\rightarrow T_{p_1}M_1\oplus T_{p_2}M_2,$ defined by $$\lambda(x) = (\pi_{1*}, \pi_{2*})x,$$ is an isomorphism.
\end{proposition}
\begin{proposition}\label{Curvature}\cite{chenbook}
Let $M = M_1\times_{f}M_2$ be a warped product of two Riemannian manifolds. If $X_1, Y_1, Z_1 \in\mathcal{L}(M_1)$ and $X_2, Y_2, Z_2 \in\mathcal{L}(M_2),$  then we have
\begin{enumerate}
	\item $R(X_1, Y_1)Z_1 \in\mathcal{L}(M_1)$ is the lift of $R^1(X_1, Y_1)Z_1$ on $M_1,$
	\item $R(X_1, Y_2 )Z_1 = \frac{H^{f}(X_1, Y_1)}{f}Y_2,$
	\item $R(X_1, Y_1)Y_2 = R(Y_2, Z_2)X_1 = 0,$
	\item $R(X_1, Y_2)Z_2 = -\frac{g_M(Y_2, Z_2)}{f}\nabla_{X_1}\nabla f,$
	\item $R(X_2, Y_2)Z_2 = R^2(X_2, Y_2)Z_2+\frac{||\nabla f||^2}{f^2}\{g_M(X_2, Z_2)Y_2-g_M(Y_2, Z_2)X_2\},$
	where $R, R^1, R^2$ are the Riemannian curvature tensors of $M, M_1, M_2,$ respectively, and $H^f$
	is the Hessian of $f.$
\end{enumerate}
\end{proposition}
\section{Clairaut Riemannian warped product maps}
\begin{theorem}\label{thm1}
Let $\phi ( = \phi_1\times \phi_2): M (= M_1\times_fM_2)\rightarrow N (= N_1\times_\rho N_2)$ be a Riemannian warped product map between two Riemannian warped product manifolds and $\gamma = (\gamma_1, \gamma_2)$ be a regular curve on $M,$ where $\gamma_1$ and $\gamma_2$ are projections of $\gamma$ onto $M_1$ and $M_2,$ respectively. Then, for $\dot{\gamma} = (\dot{\gamma}_1, \dot{\gamma}_2)$ having $\mathcal{V}_1\dot{\gamma}_1 = U_1, \mathcal{H}_1\dot{\gamma}_1 = Y_1,$
 and $\mathcal{V}_2\dot{\gamma}_2 = U_2, \mathcal{H}_2\dot{\gamma}_2 = Y_2,$ we have
 \begin{enumerate}
\item[(i)] If $\dot{\gamma}$ is a vertical vector field, then $\gamma$ is geodesic if and only if
 \begin{equation}\label{gv}
 	\begin{split}
&\hat{\nabla}^1_{U_1}U_1+2\frac{U_1(f)}{f}U_2+\hat{\nabla}^2_{U_2}U_2-g_M(U_2, U_2)(\mathcal{V}\nabla^M\ln f) = 0, \\& \text{and}~~
 T_1(U_1, U_1)+T_2(U_2, U_2)-g_M(U_2, U_2)(\mathcal{H}\nabla^M\ln f)= 0.
\end{split}
\end{equation}
\item [(ii)] If $\dot{\gamma}$ is a horizontal vector field, then $\gamma$ is geodesic if and only if
\begin{equation}\label{hor}
\begin{split}
&A_1(Y_1, Y_1)+A_2(Y_2, Y_2)-g_M(Y_2, Y_2)(\mathcal{V}\nabla^M \ln f) = 0, \\& \text{and}~~
\mathcal{H}_1\nabla^{M_1}_{Y_1}Y_1+2\frac{Y_1(f)}{f}Y_2+\mathcal{H}_2\nabla^{M_2}_{Y_2}Y_2-g_M(Y_2, Y_2)(\mathcal{H}\nabla^M \ln f) = 0.
\end{split}		
\end{equation}
\item [(iii)] If $\dot{\gamma}$ is an arbitrary vector field, then $\gamma$ is geodesic if and only if
\begin{equation}\label{mix}
	\begin{split}
&\hat{\nabla}^1_{U_1}U_1+T_1(U_1, Y_1)+\mathcal{V}_1\nabla^{M_1}_{Y_1}U_1+A_1(Y_1, Y_1)+2\frac{U_1(f)}{f}U_2\\&+2\frac{Y_1(f)}{f}U_2-g_M(U_2, U_2)(\mathcal{V}\nabla^M \ln f)-g_M(Y_2, Y_2)(\mathcal{V}\nabla^M \ln f)\\&\hat{\nabla}^2_{U_2}U_2+T_2(U_2, Y_2)+A_2(Y_2, Y_2)+\mathcal{V}_2\nabla^{M_2}_{Y_2}U_2= 0, \\&\text{and}~~
T_1(U_1, U_1)+T_2(U_2, U_2)+\mathcal{H}_1\nabla^{M_1}_{U_1}Y_1+\mathcal{H}_2\nabla^{M_2}_{U_2}Y_2+A_1(Y_1, U_1)\\&+A_2(Y_2, U_2)+\mathcal{H}_1\nabla^{M_1}_{Y_1}Y_1+\mathcal{H}_2\nabla^{M_2}_{Y_2}Y_2+2\frac{U_1(f)}{f}Y_2+2\frac{Y_1(f)}{f}Y_2\\&-g_M(U_2, U_2)(\mathcal{H}\nabla^M \ln f)-g_M(Y_2, Y_2)(\mathcal{H}\nabla^M \ln f) = 0.
\end{split}
\end{equation}
\end{enumerate} 
\end{theorem}
\begin{proof}
Let $\gamma : I\rightarrow M$ be a regular curve. Then 
$$\nabla^M_{\dot{\gamma}}\dot{\gamma} = \nabla^M_{(\dot{\gamma}_1, \dot{\gamma}_2)}(\dot{\gamma}_1, \dot{\gamma}_2).$$
Making use of Proposition \ref{Ismphm}, we can write
$$\nabla^M_{\dot{\gamma}}\dot{\gamma} =  \nabla^M_{\dot{\gamma}_1}\dot{\gamma}_1+\nabla^M_{\dot{\gamma}_1}\dot{\gamma}_2+\nabla^M_{\dot{\gamma}_2}\dot{\gamma}_1+\nabla^M_{\dot{\gamma}_2}\dot{\gamma}_2.$$
Using  Lemma \ref{lemma}, we get
\begin{equation}\label{dotgamma}
\nabla^M_{\dot{\gamma}}\dot{\gamma} = \nabla^{M_1}_{\dot{\gamma}_1}\dot{\gamma}_1+2\frac{\dot{\gamma}_1(f)}{f}\dot{\gamma}_2-g_M(\dot{\gamma}_2, \dot{\gamma}_2)(\nabla^M\ln f) +\nabla^{M_2}_{\dot{\gamma}_2}\dot{\gamma}_2.
\end{equation}
\begin{enumerate}
\item [(i)] Let $\dot{\gamma}$ be a vertical vector field, i.e., $\dot{\gamma}_1 = U_1$ and $\dot{\gamma}_2 = U_2.$ Then, from \eqref{dotgamma}, we have
$$\nabla^M_{\dot{\gamma}}\dot{\gamma} = \nabla^{M_1}_{U_1}U_1+2\frac{U_1(f)}{f}U_2-g_M(U_2, U_2)(\nabla^M\ln f) +\nabla^{M_2}_{U_2}U_2.$$ 
Using Lemma \ref{Neill tensor} in above equation, we get
\begin{equation}\label{Vdotgamma}
\begin{split}
\nabla^M_{\dot{\gamma}}\dot{\gamma} &= T_1(U_1, U_1)+\hat{\nabla}^1_{U_1}U_1+2\frac{U_1(f)}{f}U_2+T_2(U_2, U_2)+\hat{\nabla}^2_{U_2}U_2\\&-g_M(U_2, U_2)(\nabla^M\ln f). 		
\end{split}
\end{equation}
Taking vertical component of vector field in \eqref{Vdotgamma}, we get
\begin{equation}\label{V1}
\mathcal{V}\nabla^M_{\dot{\gamma}}\dot{\gamma} = \hat{\nabla}^1_{U_1}U_1+2\frac{U_1(f)}{f}U_2+\hat{\nabla}^2_{U_2}U_2-g_M(U_2, U_2)(\mathcal{V}\nabla^M\ln f).
\end{equation}


Taking horizontal component of vector field in \eqref{Vdotgamma}, we get
\begin{equation}\label{H1}
\mathcal{H}\nabla^M_{\dot{\gamma}}\dot{\gamma} = T_1(U_1, U_1)+T_2(U_2, U_2)-g_M(U_2, U_2)(\mathcal{H}\nabla^M\ln f).
\end{equation}

Now, $\gamma$ is geodesic if and only if $\mathcal{V}	\nabla^M_{\dot{\gamma}}\dot{\gamma} = 0$ and $\mathcal{H}	\nabla^M_{\dot{\gamma}}\dot{\gamma} = 0.$ Therefore, from \eqref{V1} and \eqref{H1}, we get the required assertion \eqref{gv}.
	\item [(ii)] Let $\dot{\gamma}$ be a horizontal vector field, i.e., $\dot{\gamma}_1 = Y_1$ and $\dot{\gamma}_2 = Y_2.$ Then, from \eqref{dotgamma}, we have
$$\nabla^M_{\dot{\gamma}}\dot{\gamma} = \nabla^{M_1}_{Y_1}Y_1+2\frac{Y_1(f)}{f}Y_2-g_M(Y_2, Y_2)(\nabla^M\ln f) +\nabla^{M_2}_{Y_2}Y_2.$$ 
Using Lemma \ref{Neill tensor} in above equation, we get
\begin{equation}\label{Hdotgamma}
\begin{split}
\nabla^M_{\dot{\gamma}}\dot{\gamma}& = A_1(Y_1, Y_1)	+ \mathcal{H}_1\nabla^{M_1}_{Y_1}Y_1+2\frac{Y_1(f)}{f}Y_2+A_2(Y_2, Y_2)+\mathcal{H}_2\nabla^{M_2}_{Y_2}Y_2\\&-g_M(Y_2, Y_2)(\nabla^M \ln f).
\end{split}
\end{equation}
Taking vertical component of the vector field in above equation, we get
\begin{equation}\label{V2}
\mathcal{V}\nabla^M_{\dot{\gamma}}\dot{\gamma} = A_1(Y_1, Y_1)+A_2(Y_2, Y_2)-g_M(Y_2, Y_2)(\mathcal{V}\nabla^M \ln f).
\end{equation}

Taking horizontal part of vector field in \eqref{Hdotgamma}, we get
\begin{equation}\label{H2}
\mathcal{H}\nabla^M_{\dot{\gamma}}\dot{\gamma} = \mathcal{H}_1\nabla^{M_1}_{Y_1}Y_1+2\frac{Y_1(f)}{f}Y_2+\mathcal{H}_2\nabla^{M_2}_{Y_2}Y_2-g_M(Y_2, Y_2)(\mathcal{H}\nabla^M \ln f).	
\end{equation}
Now, $\gamma$ is geodesic if and only if $\mathcal{V}	\nabla^M_{\dot{\gamma}}\dot{\gamma} = 0$ and $\mathcal{H}	\nabla^M_{\dot{\gamma}}\dot{\gamma} = 0.$ Hence, from \eqref{V2} and \eqref{H2}, we get the required result \eqref{hor}.
\item [(iii)] Let $\dot{\gamma}$ be an arbitrary vector field, i.e., $\dot{\gamma}_1 =  U_1+Y_1$ and $\dot{\gamma}_2 = U_2+Y_2.$ Then,
from \eqref{dotgamma}, we obtain 
\begin{equation*}
\begin{split}
\nabla^M_{\dot{\gamma}}\dot{\gamma}& = \nabla^{M_1}_{U_1+Y_1}(U_1+Y_1)+2\frac{(U_1+Y_1)(f)}{f}(U_2+Y_2)+\nabla^{M_2}_{U_2+Y_2}U_2+Y_2\\&-g_M(U_2+Y_2, U_2+Y_2)(\nabla^M\ln f). 		
\end{split}
\end{equation*}
After simplification and using Lemma \ref{Neill tensor}, we obtain
\begin{equation}\label{VHdotGamma}
\begin{split}
\nabla^M_{\dot{\gamma}}\dot{\gamma} &= T_1(U_1, U_1)+\hat{\nabla}^1_{U_1}U_1+T_1(U_1, Y_1)+\mathcal{H}_1\nabla^{M_1}_{U_1}Y_1+A_1(Y_1, U_1)+\\&\mathcal{V}_1\nabla^{M_1}_{Y_1}U_1+A_1(Y_1, Y_1)+\mathcal{H}_1\nabla^{M_1}_{Y_1}Y_1+2\frac{U_1(f)}{f}U_2+2\frac{U_1(f)}{f}Y_2\\&+2\frac{Y_1(f)}{f}U_2+2\frac{Y_1(f)}{f}Y_2-g_M(U_2, U_2)(\nabla^Mlnf)-g_M(Y_2, Y_2)(\nabla^M\ln f)\\&+T_2(U_2, U_2)+\hat{\nabla}^2_{U_2}U_2+T_2(U_2, Y_2)+\mathcal{H}_2\nabla^{M_2}_{U_2}Y_2+A_2(Y_2, U_2)\\&+\mathcal{V}_2\nabla^{M_2}_{Y_2}U_2+A_2(Y_2, Y_2)+\mathcal{H}_2\nabla^{M_2}_{Y_2}Y_2.		
\end{split}
\end{equation}	
 Taking vertical component of vector field \eqref{VHdotGamma}, we obtain
\begin{equation}\label{V3}
	\begin{split}
		\mathcal{V}\nabla^M_{\dot{\gamma}}\dot{\gamma} =& \hat{\nabla}^1_{U_1}U_1+T_1(U_1, Y_1)+\mathcal{V}_1\nabla^{M_1}_{Y_1}U_1+A_1(Y_1, Y_1)+2\frac{U_1(f)}{f}U_2\\&+2\frac{Y_1(f)}{f}U_2-g_M(U_2, U_2)(\mathcal{V}\nabla^M \ln f)-g_M(Y_2, Y_2)(\mathcal{V}\nabla^M \ln f)\\&\hat{\nabla}^2_{U_2}U_2+T_2(U_2, Y_2)+A_2(Y_2, Y_2)+\mathcal{V}_2\nabla^{M_2}_{Y_2}U_2,
	\end{split}
\end{equation} 
and similarly, taking horizontal component of vector field \eqref{VHdotGamma}, we have
\begin{equation}\label{H3}
\begin{split}
\mathcal{H}\nabla^M_{\dot{\gamma}}\dot{\gamma}& = T_1(U_1, U_1)+T_2(U_2, U_2)+\mathcal{H}_1\nabla^{M_1}_{U_1}Y_1+\mathcal{H}_2\nabla^{M_2}_{U_2}Y_2+A_1(Y_1, U_1)\\&+A_2(Y_2, U_2)+\mathcal{H}_1\nabla^{M_1}_{Y_1}Y_1+\mathcal{H}_2\nabla^{M_2}_{Y_2}Y_2+2\frac{U_1(f)}{f}Y_2+2\frac{Y_1(f)}{f}Y_2\\&-g_M(U_2, U_2)(\mathcal{H}\nabla^M \ln f)-g_M(Y_2, Y_2)(\mathcal{H}\nabla^M \ln f).
\end{split}
\end{equation}
Now, $\gamma$ is geodesic if and only if $\mathcal{V}\nabla^M_{\dot{\gamma}}\dot{\gamma} = 0$ and $\mathcal{H}\nabla^M_{\dot{\gamma}}\dot{\gamma} = 0.$ Therefore, from \eqref{H3} and \eqref{V3}, we obtain the required assertion \eqref{mix}.\\
This completes the proof.
\end{enumerate}
\end{proof}
\begin{theorem}
Let $\phi (= \phi_1\times\phi_2) : M (= M_1\times_fM_2)\rightarrow N (= N_1\times_\rho N_2)$ be a Riemannian warped product map between two Riemannian warped product manifolds. Then $\phi$ is a Clairaut Riemannian warped product map with $r = e^g$ if and only if $\phi_1$ has totally umbilical fibers and $\phi_2$ has totally geodesic fibers and the mean curvature vector field of $\phi$ is $-\nabla^M \ln f$, where $\phi_i:M_i\rightarrow N_i$ are Riemannian maps and $g$ is a smooth function on $M.$
\end{theorem}
\begin{proof}
	Let $\gamma : I\rightarrow M$ be a geodesic on $M$ with $\mathcal{V}\dot{\gamma}(t) = U(t) = (U_1(t), U_2(t))$ and $\mathcal{H}\dot{\gamma}(t) = Y(t) = (Y_1(t), Y_2(t))$ and
	$\omega(t)$ denotes the angle lying in $[0, \frac{\pi}{2}]$ between $\dot{\gamma}(t)$ and $Y(t).$ Assuming $b = ||\dot{\gamma}(t)||^2,$ we can write
	\begin{equation}\label{cos sqr}
		g_{\gamma(t)}(Y(t), Y(t)) = b \cos^2\omega(t),
	\end{equation}
	\begin{equation}\label{sin sqr}
		g_{\gamma(t)}(U(t), U(t)) = b \sin^2\omega(t).
	\end{equation}
	Taking derivative of \eqref{cos sqr} w.r.t `t', we get
	$$\frac{d}{dt}g_{\gamma(t)}(Y(t), Y(t)) = -2b\cos\omega(t) \sin\omega(t)\frac{d\omega}{dt}$$
	which gives
	$$g_M(\nabla^M_{\dot{\gamma}}Y, Y) = g_M(\nabla^M_{X_1+X_2}(Y_1+Y_2), Y_1+Y_2) = -b\cos\omega(t)\sin\omega(t)\frac{d\omega}{dt}$$ 
	which gives
	\begin{equation*}
		\begin{split}
&g_M(\nabla^M_{U_1}Y_1+\nabla^M_{U_2}Y_1+\nabla^M_{Y_1}Y_1+\nabla^M_{Y_2}Y_1+\nabla^M_{U_1}Y_2+\nabla^M_{U_2}Y_2+\nabla^M_{Y_1}Y_2+\nabla^M_{Y_2}Y_2, Y_1+Y_2)\\& = -b\cos\omega(t)\sin\omega(t)\frac{d\omega}{dt}.			
		\end{split}
	\end{equation*}
	Using Lemma  \ref{Neill tensor} and Lemma \ref{lemma} in above equation, we get 
	\begin{equation*}
		\begin{split}
		&g_M(T(U_1, Y_1)+H\nabla^M_{U_1}Y_1+\frac{Y_1(f)}{f}U_2+A(Y_1, Y_1)+H\nabla^M_{Y_1}Y_1+\frac{Y_1(f)}{f}Y_2+\frac{U_1(f)}{f}Y_2\\&+T(U_2, Y_2)+H\nabla^M_{U_2}Y_2+\frac{Y_1(f)}{f}Y_2+A(Y_2, Y_2)+H\nabla^M_{Y_2}Y_2, Y)= -b\cos\omega(t)\sin\omega(t)\frac{d\omega}{dt}		
		\end{split}
	\end{equation*}
	which gives
	\begin{equation*}
		\begin{split}
		&g_M(H\nabla^M_{U_1}Y_1+H\nabla^M_{Y_1}Y_1+\frac{Y_1(f)}{f}Y_2+\frac{U_1(f)}{f}Y_2+\frac{Y_1(f)}{f}Y_2+H\nabla^M_{Y_2}Y_2+H\nabla^M_{U_2}Y_2, Y) \\&= -b\cos\omega(t)\sin\omega(t)\frac{d\omega}{dt}.
		\end{split}
	\end{equation*}
Using Theorem \ref{thm1} in above equation, we get
\begin{equation*}
\begin{split}
&g_M\big(-\frac{U_1(f)}{f}Y_2-\frac{Y_1(f)}{f}Y_2-\frac{Y_1(f)}{f}Y_2-A(Y_1, U_1)-A(Y_2, U_2)-A(Y_1, U_2)\\&-A(Y_2,U_1)-T(U_1, U_1)-T(U_2, U_2)+\frac{Y_1(f)}{f}Y_2+\frac{U_1(f)}{f}Y_2+\frac{Y_1(f)}{f}Y_2, Y\big)\\& =-b\cos\omega(t)\sin\omega(t)\frac{d\omega}{dt}
\end{split}
\end{equation*}
which implies
\begin{equation}\label{T3.22}
g_M(T(U, U), Y) = b\cos\omega(t)\sin\omega(t)\frac{d\omega}{dt}.
\end{equation}
	Since $\phi$ is a Clairaut Riemannian warped product map with $r = e^g$ if and only if $\frac{d}{dt}(e^{g\circ\gamma}\sin\omega t) = 0$ this implies
	$e^{g\circ\gamma}\cos\omega(t)\frac{d\omega}{dt}+e^{g\circ\gamma}\sin\omega (t)\frac{d}{dt}(go\gamma) = 0.$\\
	Multiplying by $b\sin\omega(t)$ and using \eqref{sin sqr}, we get
	$$g_M(U, U)g_M(\dot{\gamma}, \nabla^M g) = -b\cos\omega(t)\sin\omega(t)\frac{d\omega}{dt}.$$
	From above equation and \eqref{T3.22}, we get
	$$g_M(T(U, U), Y) = -g_M(U, U)g_M(\dot{\gamma}, \nabla^M g)$$
For any geodesic $\gamma$ on $M,$ if an initial tangent vector is vertical, then $\nabla^M g$ turns out to be horizontal this implies $g_M(U, \nabla^M g) = 0.$	Therefore, from above equation,  we gwt
	$$T(U, U) = -g_M(U, U)\nabla^M g.$$
	Then for any $U_1\in\Gamma(ker\phi_{1*})$ and $U_2\in\Gamma(ker\phi_{2*}),$ we have
	\begin{equation*}
		\begin{split}
	T(U_1, U_1) = & T_1(U_1, U_1) = -g_{M_1}(U_1, U_1)\nabla^Mg,\\
		T(U_2, U_2) =& T_2(U_2, U_2)-g_{M}(U_2, U_2)\nabla^M \ln f = -g_{M}(U_2, U_2)\nabla^M g.		
		\end{split}
	\end{equation*}
This shows that $\phi$ has mean curvature vector field $H = -\nabla^M g = -\nabla^M \ln f.$ We conclude that $\phi$ has totally umbilical fibers (or $\phi$ is Clairaut Riemannian warped product map) if and only if $\phi_1$ has totally umbilical fibers and $\phi_2$ has totally geodesic fibers. 
		\end{proof}
\subsection{Curvatures, Ricci soliton and Bochner type formulae}
\begin{theorem}\label{Ricci curv.}
Let $\phi ( = \phi_1 \times \phi_2) : M (= M_1\times_f M_2)\rightarrow N (= N_1 \times_\rho N_2)$ be a Clairaut Riemannian
warped product map between two Riemannian warped product
manifolds where $\dim M_1 = m_1, \dim M_2 = m_2, \dim N_1 = n_1$ and $\dim N_2 = n_2.$ Then, the Ricci curvatures are given as
\begin{enumerate}
\item [(i)]\begin{equation*}
\begin{split}
\Ric(U_1, V_1) = &\hat{Ric}^1(U_1, V_1)-(m_1-n_1)||\nabla^M\ln f||^2g_{M_1}(U_1, V_1)-m_2fH^f(U_1, V_1)\\&-g_{M_1}(U_1, V_1)\Div_{h_1}(\nabla^M\ln f)+g_{M_1}\big(A_1(e_a, U_1), A_1(e_a, V_1)\big),
\end{split}		
\end{equation*} 
\item [(ii)]\begin{equation*}
\begin{split}
\Ric(U_2, V_2) = &f^2\hat{Ric}^2(U_2, V_2)-(f{\Delta f}+f^2(m_2-1)||\grad f||^2)g_{M_2}(U_2, V_2)\\& +f^2g_{M_2}\big(A_2(\tilde{e_b}, U_2), A_2(\tilde{e_b}, V_2)\big),
\end{split}
\end{equation*}
 \item [(iii)]\begin{equation*}
	\begin{split}
		\Ric(Y_1, U_1) =& -(m_1-n_1-1)g_{M_1}(\nabla^{M_1}_{U_1}\nabla^M \ln f, Y_1)+g_{M_1}\big((\nabla^{M_1}_{e_a}A_1)_{e_a}Y_1, U_1\big)\\&+2g_{M_1}(T_1(U_1, e_a), A_1(e_a, Y_1))	-fg_{M_1}(\nabla^{M_1}_{Y_1}\nabla f, U_1)m_2,
	\end{split}
\end{equation*}
\item [(iv)] \begin{equation*}
	\Ric(Y_2, U_2) =f^2g_{M_2}((\nabla^{M_2}_{\tilde{{e_b}}}A_2)_{\tilde{{e_b}}}Y_2, U_2),
\end{equation*}
\item [(v)]\begin{equation*}
\begin{split}
\Ric(Y_1, Z_1) = &\Ric^{{\text{range}}\phi_{1*}}(\phi_{1*}Y_1, \phi_{1*}Z_1)-(m_1-n_1)g_{M_1}(\nabla^{M_1}_{Y_1}\nabla^M\ln f, Z_1)\\&+g_{M_1}\big((\nabla^{M_1}_{e_i}A_1)_{Y_1}Z_1, e_i\big)-g_{M_1}\big(T_1(e_i, Y_1), T_1(e_i, Z_1)\big)\\&+g_{M_1}\big(A_1(Y_1, e_i), A_1(Z_1, e_i)\big)+g_{N_1}\big((\nabla\phi_{1*})(Y_1, Z_1), \tau^{(\ker\phi_{1*})^\perp}\big)\\&-m_2fH^f(Y_1, Z_1)-g_{N_1}\big((\nabla\phi_{1*})(Y_1, e_a), (\nabla\phi_{1*})(Z_1, e_a)\big),
\end{split}
\end{equation*}
\item [(vi)]\begin{equation*}
\begin{split}
\Ric(Y_2, Z_2) = &f^2\Ric^{range\phi_{2*}}(\phi_{2*}Y_2, \phi_{2*}Z_2)+f^2g_{N_2}\big((\nabla\phi_{2*})(Y_2, Z_2), \tau^{(\ker\phi_{2*})^\perp}\big)\\&+f^2g_{M_2}(A_2(Y_2, \tilde{e_j}), A_2(Z_2, \tilde{e_j})-f^2g_{N_2}\big((\nabla\phi_{2*})(Y_2, \tilde{e_b}), (\nabla\phi_{2*})(Z_2, \tilde{e_b})\big)\\&+f^2g_{M_2}\big((\nabla^{M_2
}_{\tilde{e_j}}A_2)_{Y_2}Z_2, \tilde{e_j})-(f{\Delta f}+f^2(m_2-1)||\grad f||^2)g_{M_2}(Y_2, Z_2),		
\end{split}
\end{equation*}
\item [(vii)] \begin{equation*}
\Ric(Y_1, Y_2) = \Ric(Y_1, U_2) = \Ric(Y_2, U_1) = \Ric(U_1, U_2) = 0,
\end{equation*}
\end{enumerate}
where $\hat{Ric}^1, \hat{Ric}^2, Ric^{range\phi_{1*}}, Ric^{range\phi_{2*}}$ denote the Ricci curvature of fibres of $\phi_1,$ fibres of $\phi_2,$ $\text{range}~\phi_{1*},$ $\text{range}~\phi_{2*},$ respectively and $Y_i, Z_i\in\Gamma(\mathcal{H}_i), U_i, V_i\in\Gamma(\mathcal{V}_i), i=1,2.$
\end{theorem}
\begin{proof}
Consider the orthonormal basis  
$\{e_1,\ldots,e_{m_1-n_1},e_{m_1-n_1+1},\ldots,e_{m_1}\}$ of $\mathcal{L}(M_1),$ 
where $\{e_i:1\leq i\leq m_1-n_1\}$ and $\{e_a:m_1-n_1+1\leq a\leq m_1\}$ are the vertical and horizontal components of lift of vertical vector fields on $M_1.$ In similar way, we define the orthonormal basis of $\mathcal{L}(M_2)$  as
$\{\tilde{e_1},\ldots,\tilde{e}_{m_2-n_2}, \tilde{e}_{m_2-n_2+1},\ldots,\tilde{e}_{m_2}\},$ 
where $\{\tilde{e_j}:1\leq j\leq m_2-n_2\}$ and $\{\tilde{e_b}:m_2-n_2+1\leq b\leq m_2\}$ are the vertical and horizontal components. Then, we have
	\item \begin{equation}
	\begin{split}
		\Ric(U_1, V_1) = &\sum_{i =1}^{m_1-n_1}g_M(R(e_i, U_1)V_1, e_i)+\sum_{a =m_1-n_1+1}^{m_1}g_M(R(e_a, U_1)V_1, e_a)\\&+\sum_{j=1}^{m_2-n_2}g_M(R(\tilde{e_j}, U_1)V_1, \tilde{e_j})+\sum_{b =m_2-n_2+1}^{n_2}g_M(R(\tilde{{e_b}}, U_2) V_2, \tilde{{e_b}}).
	\end{split}
\end{equation}
Using Proposition \ref{Curvature}, we get
\begin{equation*}
	\begin{split}
		\Ric(U_1, V_1) = &\Ric^1(U_1, V_1)-\frac{1}{f}\sum_{j=1}^{m_2-n_2}g_M(\tilde{e_j}, \tilde{e_j})g_M(\nabla^{M_1}_{U_1}\grad f, {V_1})\\&-\frac{1}{f}\sum_{b = m_2-n_2+1}^{m_2}g_M(\tilde{e_b}, \tilde{e_b})g_M(\nabla^{M_1}_{U_1}\grad f, {V_1}),
	\end{split}
\end{equation*}
where $Ric^1(U_1, V_1)$ is the Ricci curvature of $M_1.$ This further reduces to
\begin{equation}
	\begin{split}
		\Ric(U_1, V_1) = &\Ric^1(U_1, V_1)-fm_2H^f(U_1, V_1).
	\end{split}
\end{equation}
Since $\phi$ is a Clairaut Riemannian warped product map, using the Clairaut condition in Lemma 3.1 of \cite{sahinRmaps}, we get 
\begin{equation}
	\begin{split}
		\Ric(U_1, V_1) = &\hat{\Ric}^1(U_1, V_1)-(m_1-n_1)||\nabla^M\ln f||^2g_{M_1}(U_1, V_1)-m_2fH^f(U_1, V_1)\\&-g_{M_1}(U_1, V_1)\Div(\nabla^M\ln f)+g_{M_1}\big(A_1(e_a, U_1), A_1(e_a, V_1)\big)\\&-fm_2H^f(U_1, V_1),
	\end{split}
\end{equation}
which proves $(i).$ For $(ii),$ we proceed as follows.\\
We have
\begin{equation}
	\begin{split}
		\Ric(U_2, V_2) = &\sum_{i =1}^{m_1-n_1}g_M(R(e_i, U_2)V_2, e_i)+\sum_{a= m_1-n_1+1}^{m_1}g_M(R(e_a, U_2)V_2, e_a)\\&+\sum_{j=1}^{m_2-n_2}g_M(R(\tilde{e_j}, U_2)V_2, \tilde{e_j})+\sum_{b =m_2-n_2+1}^{n_2}g_M(R(\tilde{{e_b}}, U_2) V_2, \tilde{{e_b}}).
	\end{split}
\end{equation}
Using Proposition \eqref{Curvature} in above equation, we get
\begin{equation}
	\begin{split}
		\Ric(U_2, V_2) = &f^2\Ric^2(U_2, V_2)-\Big(f\Delta f+f^2(m_2-1)||\nabla f||^2\Big)g_{M_2}(U_2, V_2),
	\end{split}
\end{equation}
where $\Ric^2(U_2, V_2)$ is the Ricci curvatres of $M_2$ with $g_{M_2}.$
Since $\phi$ is a Clairaut Riemannian warped product map, $\phi_{2}$ is a Riemannian map with totally geodesic fibres. Using this in Lemma 3.1 of \cite{sahinRmaps}, we get $(ii).$\\
Proceeding in a similar manner, we obtain $(iii)$ to $(vi).$\\
Next, we prove $(vii),$ we have
\begin{equation}
\begin{split}
\Ric(Y_1, Y_2) = &\sum_{i =1}^{m_1-n_1}g_M(R(e_i, Y_1)Y_2, e_i)+\sum_{a = m_1-n_1+1}^{m_1}g_M(R(e_a, Y_1)Y_2, e_a)\\&+\sum_{j = 1}^{m_2-n_2}g_M(R(\tilde{e_j}, Y_1)Y_2, \tilde{e_j})+\sum_{b = m_2-n_2+1}^{m_2}g_M(R(\tilde{{e_b}}, Y_1)Y_2, \tilde{{e_b}}). 
\end{split}
\end{equation}
Using Proposition \ref{Curvature} in above equation, we get $$\Ric(Y_1, Y_2) = 0.$$
In a similar manner, we obtain $$\Ric(Y_1, U_2) = \Ric(Y_2, U_1) = \Ric(U_1, U_2) = 0.$$

\end{proof}
\begin{theorem}
Let $\phi ( = \phi_1 \times \phi_2) : M (= M_1\times_f M_2)\rightarrow N (= N_1 \times_\rho N_2) $ be a Clairaut Riemannian
warped product map whose total manifold admits Ricci soliton with vertical potential vector field $\xi$ such that $(ker\phi_*)^\perp$ is integrable. Then the fibres of $\phi_{1}$ and $\phi_2$ admit an almost Ricci soliton with potential vector field $(\xi_1-m_2f\nabla f)$ and $\xi_2,$ respectively.
\end{theorem}
\begin{proof}
Let $(M, g_M, \xi, \rho)$ be a Ricci soliton. Therefore
\begin{equation}\label{Ricci eqn vrt}
\frac{1}{2}(L_{\xi}g_M)(U, V) + \Ric(U, V)+\rho g_M(U, V) = 0,
\end{equation}
where $U, V, \xi\in\Gamma(ker\phi_*)$ and
$U = (U_1, U_2), V = (V_1, V_2), \xi = (\xi_1, \xi_2).$
One can easily show that
\begin{align}
(L_{\xi}g_M)(U, V) =& (L_{\xi_1}g_{M_1})(U_1, V_1)+f^2(L_{\xi_2}g_{M_2})(U_2, V_2)+2f\xi_1(f)g_{M_2}(U_2, V_2), \nonumber \\
\Ric(U, V) =& \Ric(U_1, V_1)+\Ric(U_2, V_2)+\Ric(U_2, V_1)+\Ric(U_1, V_2),\nonumber\\
g_M(U, V) =& g_{M_1}(U_1, V_1)+f^2g_{M_2}(U_2, V_2).	\nonumber
\end{align}
Making use of above equations in \eqref{Ricci eqn vrt}, we get
\begin{equation}\label{4.24}
\begin{split}
&\frac{1}{2} (L_{\xi_1}g_{M_1})(U_1, V_1)+\frac{1}{2} f^2(L_{\xi_2}g_{M_2})(U_2, V_2)+f\xi_1(f)g_2(U_2, V_2)+\Ric(U_1, V_1)\\&+\Ric(U_2, V_2)+\Ric(U_1, V_2)+\Ric(U_2, V_1)+\rho g_{M_1}(U_1, V_1)+\rho f^2g_{M_2}(U_2, V_2)= 0.		
\end{split}
	\end{equation}
	Using Theorem \ref{Ricci curv.} in \eqref{4.24}, we obtain
	\begin{equation}\label{4.25}
		\begin{split}
			&\frac{1}{2} (L_{\xi_1}g_{M_1})(U_1, V_1)+\frac{1}{2} f^2(L_{\xi_2}g_{M_2})(U_2, V_2)+f\xi_1 (f)g_{M_2}(U_2, V_2)+\hat{\Ric}^1(U_1, V_1)\\&-(m_1-n_1)||\nabla^M \ln f||^2g_{M_1}(U_1, V_1)-g_{M_1}(U_1, V_1)\Div_{h_1}(\nabla^M \ln f)+f^2\hat{\Ric}^2(U_2, V_2)\\&+g_{M_1}\big(A_1(e_a, U_1) A_1(e_a, V_1)\big)-m_2fH^f(U_1, V_1)-\big(f\Delta f+f^2(m_2-1)||\nabla f||^2\big)\\&g_{M_2}(U_2, V_2) +f^2g_{M_2}\big(A_2(\tilde{e_b}, U_2), A_2(\tilde{e_b}, V_2)\big)+\rho g_{M_1}(U_1, V_1)+\rho f^2g_{M_2}(U_2, V_2) = 0.		
		\end{split}
	\end{equation}
	Since $(ker\phi_*)^\perp$ is integrable, \eqref{4.25} reduces to
	\begin{equation*}
		\begin{split}
			&\frac{1}{2} (L_{{\xi_1}-m_2f\nabla f}g_{M_1})(U_1, V_1)+f^2\frac{1}{2}(L_{\xi_2}g_{M_2})(U_2, V_2)+f\xi_1 (f)g_2(U_2, V_2)+\hat{\Ric}^1(U_1, V_1)\\&-(m_1-n_1)||\nabla^M \ln f||^2g_{M_1}(U_1, V_1)-g_{M_1}(U_1, V_1)\Div_{h_1}(\nabla^M lnf)+f^2\hat{\Ric}^2(U_2, V_2)\\&+\rho g_{M_1}(U_1, V_1)+\rho f^2g_{M_2}(U_2, V_2)-\big(f\Delta f+f^2(m_2-1)||\nabla f||^2\big)g_{M_2}(U_2, V_2)\\&+m_2||\mathcal{V}\nabla f||^2g_{M_1}(U_1, V_1) = 0		
		\end{split}
	\end{equation*}
	which proves the assertion.
\end{proof}
\begin{theorem}
Let $\phi ( = \phi_1 \times \phi_2) : M (= M_1\times_f M_2)\rightarrow N (= N_1 \times_\rho N_2)$ be a totally geodesic Clairaut Riemannian warped product map whose total manifold admits Ricci soliton with horizontal potential vector field $\xi$ such that $(ker\phi_*)^\perp$ is integrable. Then the $range\phi_{1*}$ is an almost Ricci soliton with potential vector field $\xi_1-(\frac{m_1-n_1}{f}+m_2f)\nabla f$ and $range\phi_{2*}$ is an almost Ricci soliton with potential vector field $\xi_2$ .
\end{theorem}
\begin{proof}
Let $(M, g_M, \xi, \rho)$ be a Ricci soliton. Then
\begin{equation}\label{Ricci eqn}
\frac{1}{2}(L_{\xi}g_M)(Y, Z) + Ric(Y, Z)+\rho g_M(Y, Z) = 0,
\end{equation}
where $Y = (Y_1, Y_2), Z = (Z_1, Z_2)~ and ~ \xi = (\xi_1, \xi_2)\in\Gamma(ker\phi_*)^\perp.$
It can be casily shown that \begin{align*}
(L_{\xi}g_M)(Y, Z) = & (L_{\xi_1}g_{M_1})(Y_1, Z_1)+f^2(L_{\xi_2}g_{M_2})(Y_2, Z_2)+2f(\xi_1(f))g_{M_2}(Y_2, Z_2),\\
\Ric(Y, Z) =& \Ric(Y_1, Z_1)+\Ric(Y_2, Z_2)+\Ric(Y_1, Z_2)+\Ric(Y_2, Z_1),\\	
g_M(Y, Z) = &g_{M_1}(Y_1, Z_1)+f^2g_{M_2}(Y_2, Z_2).	
\end{align*}
Substituting above equations in \eqref{Ricci eqn}, we get
	\begin{equation}\label{4.27}
		\begin{split}
			&\frac{1}{2}(L_{\xi_1}g_{M_1})(Y_1, Z_1)+\frac{1}{2}f^2(L_{\xi_2}g_{M_2})(Y_2, Z_2) + f\xi_1(f)g_{M_2}(Y_2, Z_2) + \Ric(Y_1, Z_1)\\& + \Ric(Y_2, Z_2)+\rho g_{M_1}(Y_1, Z_1) +\rho f^2g_{M_2}(Y_2, Z_2) = 0. 
		\end{split}
	\end{equation}
	Using Theorem \ref{Ricci curv.} in \eqref{4.27}, we get
	\begin{equation}\label{4.28}
		\begin{split}
			&\frac{1}{2}(L_{\xi_1}g_{M_1})(Y_1, Z_1)+\frac{1}{2}f^2(L_{\xi_2}g_{M_2})(Y_2, Z_2)+f\xi_1(f)g_{M_2}(Y_2, Z_2)-m_2fH^f(Y_1, Z_1)\\&-(m_1-n_1)g_{M_1}(\nabla^{M_1}_{Y_1}\nabla^M \ln f, Z_1)+g_{M_1}\big((\nabla^{M_1}_{e_i}A_1)_{Y_1}Z_1, e_i\big)-g_{M_1}\big(T_1(e_i, Y_1),\\& T_1(e_i, Z_1)\big)+g_{M_1}\big(A_1(Y_1, e_i) A_1(Z_1, e_i)\big)+g_{N_1}\big((\nabla\phi_{1*})(Y_1, Z_1), \tau^{ker\phi_{1*}^\perp}\big)\\&-g_{N_1}\big((\nabla\phi_{1*})(Y_1, e_a), (\nabla \phi_{1*})(Z_1, e_a)\big)+f^2g_{N_2}\big((\nabla\phi_{2})_{*}(Y_2, Z_2), \tau^{\ker\phi_{2*}^\perp}\big)\\&+f^2\Ric^{range\phi_{2*}}(\phi_{2*}Y_2, \phi_{2*}Z_2)+\Ric^{range\phi_{1*}}(\phi_{1*}Y_1, \phi_{1*}Z_1)+\rho g_{M_1}(Y_1, Z_1)\\&-\big(f\Delta f+f^2(m_2-1)||\nabla f||^2\big)g_{M_2}(Y_2, Z_2)-f^2g_{N_2}\big((\nabla\phi_{2})_*(Y_2, \tilde{e_b}), \nabla\phi_{2})_*(Z_2, \tilde{e_b})\big)\\&+f^2g_{M_2}(A_2(Y_2, \tilde{e_j}), A_2(Z_2, \tilde{e_j}))+f^2g_{M_2}\big((\nabla^{M_2}_{\tilde{e_j}}A_2)_{Y_2}Z_2, \tilde{e_j})+\rho f^2g_{M_2}(Y_2, Z_2) = 0. 
		\end{split}
	\end{equation}
	Since $\phi$ is a totally geodesic Clairaut warped product
	 Riemannian map i.e., $(\nabla \phi_*) = 0$ and $(ker\phi_*)^\perp$ is an integrable distribution, i.e., $A_Y Z = 0,$ \eqref{4.28} reduces to
	\begin{equation*}
		\begin{split}
			&\frac{1}{2}(L_{(\xi_1-(m_1-n_1)\nabla^M\ln f-m_2f\nabla f )}g_{M_1})(Y_1, Z_1)+\frac{1}{2}f^2(L_{\xi_2}g_{M_2})(Y_2, Z_2)+\rho g_{M_1}(Y_1, Z_1)\\&+\Ric^{range\phi_{1*}}(\phi_{1*}Y_1, \phi_{1*}Z_1)+f^2\Ric^{range\phi_{2*}}(\phi_{2*}Y_2, \phi_{2*}Z_2)+\rho f^2g_{M_2}(Y_2, Z_2)\\&+f(\xi_1(f))g_{M_2}(Y_2, Z_2)-\big(f\Delta f+f^2(m_2-1)||\nabla f||^2\big)g_{M_2}(Y_2, Z_2)\\&-(m_1-n_1)g_{M_1}(Y_1, Z_1)||h\nabla^M \ln f||^2+m_2g_{M_1}(Y_1, Z_1)||h\nabla f||^2 = 0
		\end{split}
	\end{equation*}
which proves the desired result.	
\end{proof}
\begin{theorem}
	Let  $\phi ( = \phi_1 \times \phi_2) : M (= M_1\times_f M_2)\rightarrow N (= N_1 \times_\rho N_2)$ be a Clairaut Riemannian warped product map. Then, for any vertical vector field $V$ and horizontal vector field $Y,$ we have the following Bochner type formulae
	\begin{equation}\label{Bdiv V}
		\begin{split}
			\Div(L_Vg_M)V =& \frac{1}{2}\Delta^{\nu_1}||V_1||^2+\Ric^{ker\phi_{1*}}(V_1, V_1)+\nabla^{M_1}_{V_1}\Div_{\nu_1}(V_1)-||\nabla ^{M_1}_{\nu_1}V_1||^2\\&-2\big(V_1(f)\big)^2m_2+\frac{1}{2}\Delta^{h_1}||V_1||^2-||\nu_1\nabla^{M_1}_{h_1}V_1||^2-2fV_1(f)\Div_{\nu_2}(V_2)\\&-g_{M_1}(V_1, V_1)\Div_{h}(\nabla ln f)-g_{M_1}(\nu_1\nabla^{M_1}_{e_a}V_1, V_1)e_a(lnf)\\&+f^2\Big(\frac{1}{2}\Delta^{\nu_2}||V_2||^2+\nabla^{M_2}_{V_2}\Div_{\nu_2}(V_2)+\Ric^{\ker\phi_{2*}}(V_2, V_2)-||\nabla^{M_2}_{\nu_2}V_2||^2\\&+ \frac{1}{2}\Delta^{h_2}||V_2||^2-||\nabla^{M_2}_{h_2}V_2||^2\Big),
		\end{split}
	\end{equation}
	and 
	\begin{equation}\label{Bdiv X}
		\begin{split}
			\Div(L_Y g_M)Y = & -(m_1-n_1)Y_1(f)^2+\frac{1}{2}\Delta^{h_1}||Y_1||^2+\Ric^{(\ker\phi_{1*})^\perp}(Y_1, Y_1)+\nabla^{M_1}_{Y_1}\Div_{h_1}Y_1\\&f^2\Big(\frac{1}{2}\Delta^{h_2}||Y_2||^2-||\nabla^{M_2}_{h_2}Y_2||^2+\Ric^{(\ker\phi_{2*})^\perp}(Y_2, Y_2)+\nabla^{M_2}_{Y_2}\Div_{h_2}Y_2\Big)\\&-||\nabla^{M_1}_{h_1}Y_1||^2 -2(Y_1(f))^2m_2-2fY_1(f)\Div_{h_2}Y_2,
		\end{split}
	\end{equation}
where $\Delta^{v_1}, \Delta^{v_2}, \Delta^{h_1}$ and $\Delta^{h_2}$ are the Laplace operators on the vertical and horizontal space of $M_1$ and $M_2,$ respectively.
\end{theorem}
\begin{proof}
	Let $V = (V_1, V_2)$ be a vertical vector field with parallel orthonormal frames in a neighbourhood of a point $p =(p_1, p_2)\in M.$ Then, one can write
	\begin{equation}\label{div V}
		\Div(L_Vg_M)V = \Div(L_{V_1}g_M)V_1+\Div(L_{V_1}g_M)V_2+\Div(L_{V_2}g_M)V_1+\Div(L_{V_2}g_M)V_2.		
	\end{equation}
	Now,
	\begin{equation}\label{div V_1}
		\begin{split}
			\Div(L_{V_1}g_M)V_1 = &(\nabla^M_{e_i}L_{V_1}g_M)(V_1, e_i)+(\nabla^M_{e_a}L_{V_1}g_M)(V_1, e_a)\\&+(\nabla^M_{\tilde{e_j}}L_{V_1}g_M)(V_1, {\tilde{e_j}})+(\nabla^M_{\tilde{e_b}}L_{V_1}g_M)(V_1, {\tilde{e_b}}).
		\end{split}
	\end{equation} 
	Further, $$(\nabla^M_{e_i}L_{V_1}g_M)(V_1, e_i) = \nabla^M_{e_i}(L_{V_1}g_M(V_1, e_i))-L_{V_1}g_M(
	\nabla^M_{e_i}V_1, e_i)-L_{V_1}g_M(\nabla^M_{e_i}e_i, V_1)$$
	as $\{e_i\}$ is a orthonormal frame in a neighbourhood of $p$ which is parallel at $p.$ After some computations, we get
	\begin{equation*}
		\begin{split}
			(\nabla^M_{e_i}L_{V_1}g_M)(V_1, e_i) = & \nabla^M_{e_i}(g_M(V_1, \nabla^M_{e_i}{V_1}))+g_M(\nabla^M_{e_i}\nabla_{V_1}V_1, e_i)-g_M(\nabla^M_{\nabla^M_{e_i}V_1}V_1, e_i)\\&-g_M(\nabla^M_{e_i}V_1, \nabla^M_{e_i}V_1)		
		\end{split}
	\end{equation*}
	which implies
	\begin{equation}\label{V_1}
		(\nabla^M_{e_i}L_{V_1}g_M)(V_1, e_i) = \frac{1}{2}\Delta^{\nu_1}||V_1||^2+\Ric^{\ker\phi_{1*}}(V_1, V_1)+\nabla^{M_1}_{V_1}\Div_{\nu_1}(V_1)-||\nabla ^{M_1}_{\nu_1}V_1||^2.	
	\end{equation}
	By similar computations, we get
	\begin{equation}\label{H_1}
		\begin{split}
			(\nabla^M_{e_a}L_{V_1}g_M)(V_1, e_a) =& \frac{1}{2}\Delta^{h_1}||V_1||^2-||\nu_1\nabla^{M_1}_{h_1}V_1||^2-g_{M_1}(V_1, V_1)\Div_{h_1}(\nabla lnf)\\&-g_{M_1}(\nu_1\nabla^{M_1}_{e_a}V_1, V_1)e_a(lnf),
		\end{split}
	\end{equation}
	\begin{equation}\label{V_2}
		(\nabla^M_{\tilde{e_j}}L_{V_1}g_M)(V_1, {\tilde{e_j}}) = -2\big(V_1(f)\big)^2(m_2-n_2),
	\end{equation}
	and
	\begin{equation}\label{H_2}
		(\nabla^M_{\tilde{e_b}}L_{V_1}g_M)(V_1, {\tilde{e_b}}) = -2\big(V_1(f)\big)^2n_2.
	\end{equation}
	Making use of \eqref{V_1} - \eqref{H_2} in \eqref{div V_1}, we obtain
	\begin{equation}\label{LV1V1}
		\begin{split}
			\Div(L_{V_1}g_M)V_1 &= \frac{1}{2}\Delta^{\nu_1}||V_1||^2+\Ric^{\ker\phi_{1*}}(V_1, V_1)+\nabla^{M_1}_{V_1}\Div_{\nu_1}(V_1)-||\nabla ^{M_1}_{\nu_1}V_1||^2\\&-2\big(V_1(f)\big)^2m_2+\frac{1}{2}\Delta^{h_1}||V_1||^2-||\nu_1\nabla^{M_1}_{h_1}V_1||^2-g_{M_1}(V_1, V_1)\Div_{h_1}(\nabla^M \ln f)\\&-g_{M_1}(\nu_1\nabla^{M_1}_{e_a}V_1, V_1)e_a(\ln f).
		\end{split}
	\end{equation}
	By similar computations, we get
	\begin{equation}\label{LV1V2}
		\Div(L_{V_1}g_M)V_2 = 0,	
	\end{equation}
	\begin{equation}\label{LV2V1}
		\Div(L_{V_2}g_M)V_1 = -2fV_1(f)\Div_{\nu_2}(V_2),
	\end{equation}
	and
	\begin{equation}\label{LV2V2}
		\begin{split}
			\Div(L_{V_2}g_M)V_2 =& f^2\Big(\frac{1}{2}\Delta^{\nu_2}||V_2||^2+\nabla^{M_2}_{V_2}\Div_{\nu_2}(V_2)+\Ric^{\ker\phi_{2*}}(V_2, V_2)-||\nabla^{M_2}_{\nu_2}V_2||^2\\&+ \frac{1}{2}\Delta^{h_2}
			||V_2||^2-||\nabla^{M_2}_{h_2}V_2||^2\Big).	
		\end{split}
	\end{equation}
Using equations \eqref{LV1V1} to \eqref{LV2V2} and \eqref{div V}, we get \eqref{Bdiv V}.\\\\
		Let $Y = (Y_1, Y_2)$ be a horizontal vector field with parallel orthonormal frames in a neighbourhood of a point $p =(p_1, p_2)\in M.$ Then,
	\begin{equation}\label{div X}
		\Div(L_Y g_M)Y = \Div(L_{Y_1} g_M)Y_1+\Div(L_{Y_1} g_M)Y_2+\Div(L_{Y_2}g_M)Y_1+\Div(L_{Y_2}g_M)Y_2.
	\end{equation}
	Since 
	\begin{equation}\label{LY1Y1}
		\begin{split}
			\Div(L_{Y_1} g_M)Y_1 = & (\nabla^M_{e_i}L_{Y_1}g_M)(Y_1, e_i)+(\nabla^M_{e_a}L_{Y_1}g_M)(Y_1, e_a)+(\nabla^M_{\tilde{e_j}}L_{Y_1}g_M)(Y_1, \tilde{e_j})\\&+(\nabla^M_{\tilde{e_b}}L_{Y_1}g_M)(Y_1, \tilde{e_b}),
		\end{split}
	\end{equation}
	solving all the terms of above equation, we get
	\begin{equation}\label{v1X_1}
		(\nabla^M_{e_i}L_{Y_1}g_M)(Y_1, e_i) = -(m_1-n_1)Y_1(f)^2,
	\end{equation}
	\begin{equation}\label{h1X_1}
		\begin{split}
			(\nabla^M_{e_a}L_{Y_1}g_M)(Y_1, e_a) = & \frac{1}{2}\Delta^{h_1}||Y_1||^2+\Ric^{({\ker\phi_{1*}})^\perp}(Y_1, Y_1)+\nabla^{M_1}_{Y_1}\Div_{h_1}Y_1-\\&||\nabla^{M_1}_{h_1}Y_1||^2,		
		\end{split}
	\end{equation}
	\begin{equation}\label{Lv2v2}
		(\nabla^M_{\tilde{e_j}}L_{Y_1}g_M)(Y_1, \tilde{e_j}) = -2(Y_1(f))^2(m_2-n_2),
	\end{equation}
	and
	\begin{equation}\label{h2X_2}
		(\nabla^M_{\tilde{e_b}}L_{Y_1}g_M)(Y_1, \tilde{e_b}) =-2(Y_1(f))^2n_2.
	\end{equation}
	Substituting the values from \eqref{v1X_1} - \eqref{h2X_2} in \eqref{LY1Y1}, we get
	\begin{equation}\label{div X_1}
		\begin{split}
			\Div(L_{Y_1}g_M)Y_1 =& -(m_1-n_1)Y_1(f)^2+\frac{1}{2}\Delta^{h_1}||Y_1||^2+\Ric^{(\ker\phi_{1*})^\perp}(Y_1, Y_1)\\&+\nabla^{M_1}_{Y_1}\Div_{h_1}Y_1-||\nabla^{M_1}_{h_1}Y_1||^2 -2(Y_1(f))^2(m_2-n_2)-2(Y_1(f))^2n_2.	
		\end{split}
	\end{equation}
	By similar computations, we get 
	\begin{equation}\label{div X_1X_2}
		\Div(L_{Y_1}g_M)Y_2 = 0,
	\end{equation} 
	\begin{equation}\label{div X_2X_1}
		\Div(L_{Y_2}g_M)Y_1 = -2Y_1(f)f\Div_{h_2}Y_2,
	\end{equation}
	and
	\begin{equation}\label{div X_2}
		\begin{split}
			\Div(L_{Y_2}g_M)Y_2 =& f^2\Big(\frac{1}{2}\Delta^{h_2}||Y_2||^2-||\nabla^{M_2}_{h_2}Y_2||^2+\Ric^{(\ker\phi_{2*})^\perp}(Y_2, Y_2)+\nabla^{M_2}_{Y_2}\Div_{h_2}Y_2\Big).
		\end{split}
	\end{equation}
	Using equations \eqref{div X_1} to \eqref{div X_2} and \eqref{div X}, we get \eqref{Bdiv X}. 
\end{proof}
\begin{example}
	Let $M_1 = \{(x_1, x_2, x_3, x_4); x_i\in\mathbb{R}^{*}\}$ be a Riemannian manifold whose Riemannian metric is given by $g_{M_1} = e^{-2x_3}dx_1^2+e^{-2x_3}dx_2^2+dx_3^2+dx_4^2.$ Let $N_1 = \{(y_1, y_2, y_3, y_4); y_i\in\mathbb{R}, y_3\neq 0\}$ be a Riemannian manifold whose Riemannian metric is given by $g_{N_1} = dy_1^2+dy_2^2+dy_3^2+dy_4^2.$ We define a map $\phi_1:(M_1, g_{M_1})\rightarrow(N_1, g_{N_1})$ given by
	$$\phi_1(x_1, x_2, x_3, x_4) = (0, 0, e^{x_3}\cos x_4, e^{x_3}\sin x_4).$$
	After some computations, we get
	\begin{align*}
	\ker\phi_{1*} =&\text{span}~\{U_1 = e_1, U_2 = e_2\},\\
		(\ker\phi_{1*})^\perp = &span\{Y_1 = e^{x_3}\cos x_4e_3-e^{x_3}\sin x_4e_4, Y_2 = e^{x_3}\sin x_4e_3+e^{x_3}\cos x_4e_4\},
	\end{align*}
	where $\{e_1 = e^{x_3}\frac{\partial}{\partial x_1}, e_2 = e^{x_3}\frac{\partial}{\partial x_2}, e_3 = \frac{\partial}{\partial x_3}, e_4 = \frac{\partial}{\partial x_4}\}$ and $\{e_1' = \frac{\partial}{\partial y_1}, e_2' = \frac{\partial}{\partial y_2}, e_3' = \frac{\partial}{\partial y_3}, e_4' = \frac{\partial}{\partial y_4}\}$ are the bases of $T_{p_1}M_1$ and $T_{\phi_1(p_1)}N_1,$ respectively.
	One can easily prove that $\phi_1$ is a Riemannian map. After some computations, we have Christoffel symbols as 
	$$\Gamma^1_{1 3} = -1, \Gamma^2_{2 3} = -1, \Gamma^3_{1 1} = e^{-2x_3},  \Gamma^3_{2 2} = e^{-2x_3}$$ and remaining Christoffel symbols are zero. Also, we have $\nabla^{M_1}_{U_1}U_1 = \nabla^{M_1}_{U_2}U_2 = 1.$
	This implies $\phi_1$ has totally umbilical fibers.
	On the other side, let $M_2 = \mathbb{R}^2 = N_2$ be the Riemannian manifolds with Euclidean metric. Construct a map $\phi_2:M_2\rightarrow N_2,$ defined by $$\phi_2(x_5, x_6) = (\frac{x_5+x_6}{\sqrt{2}}, 0).$$
	We obtain
	$$ker\phi_{2*} = \text{span}\{V_1 = \frac{e_5-e_6}{\sqrt{2}}\}$$ and $$(ker\phi_{2*})^\perp = \text{span}\{Z_1 = \frac{e_5+e_6}{\sqrt{2}}\},$$ where $\{e_5 =\frac{\partial}{\partial x_5}, e_6 = \frac{\partial}{\partial x_6}\}$ and $\{e_5' =\frac{\partial}{\partial y_5}, e_6' = \frac{\partial}{\partial y_6}\}$ are the bases of $T_{p_2}M_2$ and $T_{\phi_{2}(p_2)}N_2,$ respectively. Clearly, $\phi_2$ is a Riemannian map. Also, $\nabla^{M_2}_{V_1}V_1 =0,$ $T_2(V_1, V_1) = 0.$  This implies $\phi_{2}$  has totally geodesic fibers.\\ Now, $M = M_1\times_f M_2$ and $ N = N_1\times_{\rho} N_2$ are the Riemannian warped product manifolds with Riemannian metric $g_M = e^{-2x_3}dx_1^2+e^{-2x_3}dx_2^2+dx_3^2+dx_4^2+dx_5^2+dx_6^2$ and $g_N = dy_1^2+dy_2^2+dy_3^2+dy_4^2+dy_5^2+dy_6^2,$ respectively. With the help of $\phi_1$ and $\phi_2$, we construct a map $\phi:M_1\times_{f}M_2\rightarrow N_1\times_{\rho}N_2$ defined by
	$$\phi(x_1, x_2, x_3, x_4, x_5, x_6) = (0, 0, e^{x_3}\cos x_4, e^{x_3}\sin x_4, \frac{x_5+x_6}{\sqrt{2}}, 0).$$
	Then, $\phi$ is a Clairaut warped product map.

\end{example}


\section{Conformal Riemannian warped product maps}
\begin{proposition}
	Let $\phi_i:M_i\rightarrow N_i,$ i = 1, 2 be a conformal Riemannian map between Riemannian manifolds with dilation $\lambda_i.$ Let $\rho:N_1\rightarrow R^+$ and $f = \rho o \phi_1$ are the functions on $N_1$ and $M_1,$ respectively. Then the map $\phi = \phi_1\times \phi_2:M(=M_1\times_fM_2)\rightarrow N (= N_1\times_\rho N_2)$ between Riemannian warped product manifolds is conformal Riemannian warped product map if the  lifts of $\lambda_i's$ on $M$ are equal.	
\end{proposition}
\begin{proof}
	Let $\phi_1:M_1\rightarrow N_1$ and $\phi_2:M_2\rightarrow N_2$ be two conformal Riemannian maps with dilation $\lambda_1$ and $\lambda_2,$ respectively. Then, for $Y_1, Z_1\in M_1,$ and $Y_2, Z_2\in M_2,$ we get
	\begin{equation*}
		\begin{split}
			g_N\Big(\phi_{*}(Y_1, Y_2), \phi_{*}(Z_1, Z_2)\Big) = & g_N\Big(\big(\phi_{1*}(Y_1), \phi_{2*}(Y_2)\big), \big(\phi_{1*}(Z_1), \phi_{2*}(Z_2)\big)\Big)\\
			&=  g_{N_1}\Big(\sigma_{1*}(\phi_{1*}(Y_1), \phi_{2*}(Y_2)), \sigma_{1*}(\phi_{1*}(Z_1), \phi_{2*}(Z_2))\Big)\\&+ \rho^2\circ\sigma_1\big(\phi_1(p_1), \phi_2(p_2)\big)g_{N_2}\Big(\sigma_{2*}(\phi_{1*}(Y_1), \phi_{2*}(Y_2)),\\& \sigma_{2*}(\phi_{1*}(Z_1), \phi_{2*}(Z_2))\Big),	
		\end{split}
	\end{equation*}
where $\sigma_1$ and $\sigma_2$ are the projections of $N_1\times_\rho N_2$ onto $N_1$ and $N_2,$ respectively. Simplifying above equation, we get
	\begin{equation*}
		\begin{split}
			g_N\Big(\phi_{*}(Y_1, Y_2), \phi_{*}(Z_1, Z_2)\Big)	= g_{N_1}\Big(\phi_{1*}(Y_1), \phi_{1*}(Z_1)\Big)+ \rho^2\circ\phi_1g_{N_2}\Big( \phi_{2*}(Y_2),  \phi_{2*}(Z_2)\Big).
		\end{split}
	\end{equation*}
Since $\phi_{1}$ and $\phi_{2}$ are conformal Riemannian maps, above equation can be written as
	\begin{equation*}
		\begin{split}
		g_N\Big(\phi_{*}(Y_1, Y_2), \phi_{*}(Z_1, Z_2)\Big)		= & \lambda_{1}^2(p_1)g_{M_1}(Y_1, Z_1)+ \rho^2(\phi_1(p_1))\lambda_{2}^2(p_2)g_{M_2}(Y_2,  Z_2)\\
			=& \lambda_{1}^2(\pi_1(p_1, p_2))g_{M_1}(Y_1, Z_1)\\&+ \rho^2(\phi_1(p_1))\lambda_{2}^2(\pi_2(p_1, p_2))g_{M_2}(Y_2,  Z_2)\\
			=& (\lambda_{1} \circ\pi_1)^2g_{M_1}(Y_1, Z_1)+ f^2(\lambda_{2}\circ\pi_2)^2g_{M_2}(Y_2,  Z_2).
		\end{split}
	\end{equation*}
	If lift of $\lambda_{1}$ and $\lambda_{2}$ are equal, i.e., $(\lambda_{1} \circ\pi_1)^2 = (\lambda_{2}\circ\pi_2)^2 = \lambda^2,$ then above equation reduces to 
	\begin{equation}
		\begin{split}
			g_N(\phi_{*}Y, \phi_{*}Z) = \lambda^2(p)g_M(Y, Z)
		\end{split}
	\end{equation}
	which shows that $\phi$ is a conformal Riemannian warped product map.
	\begin{remark}
		Every Riemannian warped product map is a conformal Riemannian warped product map with $\lambda_i = 1, 1\leq i\leq  2.$
	\end{remark}
	\begin{remark}
		If $\ker\phi_{*} = 0,$ then conformal Riemannian warped product map becomes conformal isometric immersion of Riemannian warped product.
	\end{remark}
	\begin{remark}
		If $(\text{range}~\phi_{*})^\perp = 0,$ then conformal Riemannian warped product map becomes conformal Riemannian warped product submersion.
	\end{remark}

\end{proof}
\begin{proposition}
	Let $\phi:(M, g_M) = (M_1\times_{f}M_2, g_{M_1}+f^2g_{M_2})\rightarrow (N, g_N)= (N_1\times_\rho N_2, g_{N_1}+\rho^2g_{N_2})$ be a conformal Riemannian warped product map with dilation $\mu$. Then the map $\phi:(M, g'_{M} =\mu^{2}g_M)\rightarrow(N, g_N)$ is a Riemannian warped product map.
\end{proposition}
\begin{proof}
	For any $Y, Z\in\Gamma(TM),$ we have \begin{equation*}
		\begin{split}
			g_N(\phi_{*}Y, \phi_{*}Z) =& \mu^2g_M(Y, Z)	\\
			&= \mu^2\mu^{-2}g'_{M}(Y, Z)\\
			& = g'_{M}(Y, Z)
		\end{split}
	\end{equation*}
	which proves the assertion.
\end{proof}
\begin{proposition}
	Let $\phi = (\phi_1\times \phi_2): M = M_1\times_fM_2\rightarrow N = N_1\times_\rho N_2$ be a conformal  Riemannian warped product map between Riemannian warped product manifolds. Then, we have
	\begin{equation*}
		||\phi_*||^2(p) = \lambda^2(p)(\text{rank}~ \phi_1 + \rho^2 \text{rank}~ \phi_2). 
	\end{equation*}
\end{proposition}
\begin{proof}
	Let $*\phi_{*p}$ be the adjoint of $\phi_{*p}.$ Now, we define a linear transformation $G: ((ker\phi_{*p})^\perp=ker\phi_{1*}^\perp\times ker\phi_{2*}^\perp) \rightarrow ((ker\phi_{*p})^\perp(=ker\phi_{1*}^\perp\times ker\phi_{2*}^\perp)$ by $G = *\phi_{*p} o \phi_{*p}.$ Then for any horizontal vector fields $Y = (Y_1, Y_2)$ and $Y = (Z_1, Z_2)\in\Gamma(TM),$ we have
	\begin{equation*}
		\begin{split}
		g_M(GY, Z) = &g_M(*\phi_{*p} \circ \phi_{*p} Y, Z)\\
			=& g_N(\phi_{*p}Y, \phi_{*p}Z).
		\end{split}
	\end{equation*}
Since $\phi$ is a conformal Riemannian warped product map,
	$$g_M(GY, Z) = \lambda^2(p)g_M(Y, Z).$$
This shows that
	$GY = \lambda^2(p)Y.$\\ 
 Let $\{e_k\}_{k = 1}^{k=n_1+n_2}, \{e_a\}_{a = m_1-n_1+1}^{a = m_1}$ and $\{\tilde{e_b}\}_{b = m_2-n_2+1}^{m_2}$ be the orthonormal bases of $(ker\phi_{*})^\perp, (ker\phi_{1*})^\perp$ and $(ker\phi_{2*})^\perp,$ respectively. Therefore,
 \begin{equation*}
 	\begin{split}
 ||\phi_{*}||^2 = & \sum_{k=1}^{n_1+n_2}g_N(\phi_*{e_k}, \phi_*{e_k}),\\
  &= \lambda^2(p)\Big(\sum_{a = m_1-n_1+1}^{m_1}g_{N_1}(\phi_{1*}{e_a}, \phi_{1*}{e_a}) +\rho^2(\phi_1(p_1))\sum_{b = m_2-n_2+1}^{m_2}g_{N_2}(\phi_{2*}\tilde{{e_b}}, \phi_{2*}\tilde{{e_b}})\Big).		
 	\end{split}
 \end{equation*} This implies $$\text{rank}~ \phi = \lambda^2\big(\text{rank} ~\phi_1+ \rho^2 \text{rank} ~\phi_2\big).$$
	
\end{proof}	
\begin{theorem}
	Let $\phi (= \phi_1 \times \phi_2):M(=M_1\times_f M_2)\rightarrow N(=N_1\times_\rho N_2)$ be a conformal  Riemannian warped product map with dilation $\lambda.$ Then for any horizontal vector fields $Y = (Y_1, Y_2)$, $Z = (Z_1, Z_2)\in\Gamma(TM),$ we have
	\begin{equation}
		\begin{split}
			A_Y Z = &\frac{1}{2}\{\mathcal{V}_1[Y_1, Z_1]-\lambda_1^2g_{M_1}(Y_1, Z_1)\grad_{\mathcal{V}_1}(\frac{1}{\lambda_{1}^2})\}-f^2g_{M_2}(Y_2, Z_2)\mathcal{V}\nabla^{M}\ln f\\&+\frac{1}{2}f^2\{\mathcal{V}_2[Y_2, Z_2]-\lambda_{2}^2 g_{M_2}(Y_2, Z_2)\grad_{\mathcal{V}_2}(\frac{1}{\lambda_{2}^2})\}.
		\end{split}
	\end{equation}
\end{theorem}
\begin{proof}
 Assume that $\{e_1,\ldots,e_{m_1-n_1}, \tilde{e_1},\ldots,\tilde{e}_{m_2-n_2}\}$ is a local orthonormal frame for vertical distriution $ker\phi_{*}.$ Now, 
	\begin{equation*}
		\begin{split}
			A_Y Z = A_{Y_1}Z_1 +A_{Y_2}Z_2, 
		\end{split}
	\end{equation*}
which can be written as
	\begin{equation*}
		\begin{split}
			A_Y Z = \mathcal{V}\nabla_Y Z = & g_M(\mathcal{V}\nabla^M_{Y_1}Z_1, e_i)e_i+	g_M(\mathcal{V}\nabla^M_{Y_1}Z_1, \tilde{e_j})\tilde{e_j}\\&+g_M(\mathcal{V}\nabla^M_{Y_2}Z_2, e_i)e_i+ g_M(\mathcal{V}\nabla^M_{Y_2}Z_2, \tilde{e_j})\tilde{e_j}\\
			=& g_{M_1}(\mathcal{V}_1\nabla^{M_1}_{Y_1}Z_1, e_i)e_i + g_{M_1}(\mathcal{V}_1\nabla^{M_1}_{Y_1}Z_1, \tilde{e_j})\tilde{e_j}\\&-f^2g_{M_2}(Y_2, Z_2)\mathcal{V}\nabla^{M}lnf+ f^2g_{M_2}(\mathcal{V}_2\nabla^{M_2}_{Y_2}Z_2, \tilde{e_j})\tilde{e_j}.
		\end{split}
	\end{equation*}
By using Koszul formula in above equation, we obtain
	\begin{equation*}
		\begin{split}
			A_Y Z = & \frac{1}{2}\{Y_1g_{M_1}(Z_1, e_i)+Z_1g_{M_1}(e_i, Y_1)-e_ig_{M_1}(Y_1, Z_1)-g_{M_1}(Y_1, [Z_1, e_i])\\&+g_{M_1}(Z_1, [e_i, Y_1])+g_{M_1}(e_i, [Y_1, Z_1])\}e_i-f^2g_{M_2}(Y_2, Z_2)\mathcal{V}\nabla^{M}\ln f\\&+\frac{1}{2}f^2\{Y_2g_{M_2}(Z_2, \tilde{e_j})+Z_2g_{M_2}(\tilde{e_j}, Y_2)-\tilde{e_j}g_{M_2}(Y_2, Z_2)\\&-g_{M_2}(Y_2, [Z_2, \tilde{e_j}])+g_{M_2}(Z_2, [\tilde{e_j}, Y_2])+g_{M_2}(\tilde{e_j}, [Y_2, Z_2])\}\tilde{e_j}.
		\end{split}
	\end{equation*}
After simplifying, we get
	\begin{equation}\label{Ayz}
		\begin{split}
			A_Y Z = &\frac{1}{2}\{\mathcal{V}_1[Y_1, Z_1]-\grad_{\mathcal{V}_1}(g_{M_1}(Y_1, Z_1))\}-f^2g_{M_2}(Y_2, Z_2)\mathcal{V}\nabla^{M}\ln f\\&+\frac{1}{2}f^2\{\mathcal{V}_2[Y_2, Z_2]-\grad_{\mathcal{V}_2}(g_{M_2}(Y_2, Z_2))\}.
		\end{split}
	\end{equation}
	Now, \begin{equation}\label{gradv1}
		\begin{split}
			\grad_{\mathcal{V}_1}(g_{M_1}(Y_1, Z_1)) = & \grad_{\mathcal{V}_1}(\frac{1}{\lambda_{1}^2}g_{N_1}(\phi_{1*}Y_1, \phi_{1*}Z_1))\\
			=&\lambda_1^2g_{M_1}(Y_1, Z_1)\grad_{\mathcal{V}_1}(\frac{1}{\lambda_{1}^2}).
		\end{split}
	\end{equation}
	In similar manner, we can write
	\begin{equation}\label{gradv2}
	\grad_{\mathcal{V}_2}(g_{M_2}(Y_2, Z_2)) =\lambda_{2}^2 g_{M_2}(Y_2, Z_2)\grad_{\mathcal{V}_2}(\frac{1}{\lambda_{2}^2}).
	\end{equation}
	From \eqref{Ayz}, \eqref{gradv1} and \eqref{gradv2}, we get the required result.
\end{proof}
\begin{theorem}
Let $\phi (= \phi_1 \times \phi_2):M(=M_1\times_f M_2)\rightarrow N(=N_1\times_\rho N_2)$ be a conformal Riemannian warped product map with totally umbilical fibres such that $M$ is orientable and compact. Then, 
\begin{itemize}
	\item [(i)]$\int_{M}^{}S^{\mathcal{V}_1} dv_{g_M} = \int_{M_1}^{}\Big(\hat{S}_1-(m_1-n_1)(m_1-n_1-1)||\nabla \ln f||^2\Big)dv_{g_{M_1}},$
	\item [(ii)]$\int_{M}^{}S^{\mathcal{V}_2}dv_{g_M} = \int_{M}^{}f^2(\hat{S}_2+(m_2-n_2)(1-m_2+n_2)||\nabla f||^2)dv_{g_M}$
where $S^{\mathcal{V}_1}, S^{\mathcal{V}_2}, \hat{S}_1$ and $\hat{S}_2$ denote the scalar curvatures of vertical vector field of $M_1, M_2$ and  fibres of $\phi_{1}$ and $\phi_{2},$ respectively and $dv_{g_{M_1}},$ $dv_{g_M}$ are the volume forms of $M_1$ and $M_2,$ respectively.
\end{itemize}
\end{theorem}
\begin{proof}
Let $\phi:(M, g_M)\rightarrow (N, g_N)$ be a conformal Riemannian map. Then by using the structural equations of Lemma \ref{Neill tensor}, the curvature tensor for vertical vector field is given by 
\begin{equation}\label{Riemann cur}
\begin{split}
g_M (R^M(U, V )W , E) = &g_M (\hat{R}(U, V )W , E) -  g_M \big(T(U, E), T(V, W)\big)\\&+ g_M \big(T(U, W), T(V, E)\big).
\end{split}
\end{equation}
Now, using Proposition \ref{Curvature}, Riemann curvature tensor is defined as
\begin{equation*}
\begin{split}
R(U_1, V_1, W_1, E_1) = & R^1(U_1, V_1, W_1, E_1).\\
\end{split}
\end{equation*}
Making use of \eqref{Riemann cur} in above equation, we obtain
\begin{align}
R(U_1, V_1, W_1, E_1) = &g_{M_1}(\hat{R}^1(U_1, V_1)W_1, E_1) -  g_{M_1}\big(T_1(U_1, E_1), T_1(V_1, W_1)\big)\nonumber\\&+ g_{M_1}\big(T_1(U_1, W_1), T_1(V_1, E_1)\big).
\end{align}
Again using Proposition \ref{Curvature} and \eqref{Riemann cur}, we obtain
\begin{equation}
\begin{split}
R(U_2, V_2, W_2, E_2) =&f^2\bigg(g_{M_2}(\hat{R}_2(U_2, V_2)W_2, E_2) -  g_{M_2}\big(T_2(U_2, E_2), T_2(V_2, W_2)\big)\\&+ g_{M_2}\big(T_2(U_2, W_2), T_2(V_2, E_2)\big)\bigg)+\frac{||\nabla f||^2}{f^2}\{g_M(V_2, E_2)\\&g_M(U_2, W_2)-g_M(V_2, W_2)g_M(U_2, E_2)\},
\end{split}
\end{equation}
where $U_i, V_i, W_i, E_i\in\Gamma(\mathcal{V}_i), 1\leq i\leq2.$
Therefore the scalar curvature of $S^{\mathcal{V}_1}, S^{\mathcal{V}_2}$ for conformal Riemannian warped product map are respectively given by
\begin{equation}
S^{\mathcal{V}_1} = \hat{S}_1 -g_{M_1}\big(T_1(e_i, e_i), T_1(e_k, e_k)\big)+g_{M_1}\big(T_1(e_i, e_k), T_1(e_k, e_i)\big),
\end{equation}
and
\begin{equation}
\begin{split}
S^{\mathcal{V}_2} = &f^2\bigg(\hat{S}_2 -g_{M_2}\big(T_2(\tilde{e_j}, \tilde{e_j}), T_2(\tilde{e_l}, \tilde{e_l})\big)+g_{M_2}\big(T_2(\tilde{e_j}, \tilde{e_l}), T_2(\tilde{e_l}, \tilde{e_j})\big)\bigg)\\&+||\grad f||^2f^2\{g_{M_2}(\tilde{e_l}, \tilde{e_j})g_{M_2}(\tilde{e_j}, \tilde{e_l})-g_{M_2}(\tilde{e_j}, \tilde{e_j})g_{M_2}(\tilde{e_l}, \tilde{e_l})\}.
\end{split}
\end{equation}
Since $\phi$ is a conformal Riemannian warped product map with totally umbilical fibres, $\phi_1$ is a conformal Riemannian map with totally umbilical fibres and $\phi_{2}$ is a conformal Riemannian map with totally geodesic fibres, above equations reduce to
\begin{equation}
S^{\mathcal{V}_1}= \hat{S}_1 -(m_1-n_1)(m_1-n_1-1)||\nabla^{M}\ln f||^2,
\end{equation}
and 
\begin{equation}
S^{\mathcal{V}_2} = f^2(\hat{S}_2+(m_2-n_2)(1-m_2+n_2)||\nabla f||^2).
\end{equation}
Integrating both side over $M,$ we obtain
\begin{equation}
\int_{M}^{}S^{\mathcal{V}_1} dv_{g_M} = \int_{M_1}^{}\Big(\hat{S}_1-(m_1-n_1)(m_1-n_1-1)||\nabla^M \ln f||^2\Big)dv_{g_{M_1}},
\end{equation}
and
\begin{equation}
\int_{M}^{}S^{\mathcal{V}_2}dv_{g_M} = \int_{M}^{}f^2(\hat{S}_2+(m_2-n_2)(1-m_2+n_2)||\nabla f||^2)dv_{g_M}.
\end{equation}
This completes the proof.
\end{proof}

\begin{example}
 Consider Riemannian manifolds $M_1 = R^2-\{(0, 0)\}$ with a Riemannian metric $g_{M_1} = e^{4}dx_1^2+e^{4}dx_2^2$ and $N_1= R^2$ with Euclidean metric $g_{N_1}.$ \\
 Now, we construct a map $\phi_{1}:(M_1, g_{M_1})\rightarrow (N_1, g_{N_1}),$
 defined as
 $$\phi_{1}(x_1, x_2) = \Big(\frac{x_1-x_2}{\sqrt{2}}, 0\Big).$$
 Then, $$ker\phi_{1*} = \text{span}~\{U_1 = e_1+e_2\},$$  $$(ker\phi_{1*})^\perp = \text{span}\{Y_1 = e_1-e_2\},$$
 $$range\phi_{1*}{Y_1} = span\{\sqrt{2}e^{-2}e_1'\},$$
 where $\{e_1 = e^{-2}\frac{\partial}{\partial x_1}, e_2 = e^{-2}\frac{\partial}{\partial x_2}\}$ and $\{e_1' = \frac{\partial}{\partial y_1}, e_2' = \frac{\partial}{\partial y_2}\}$ are the bases of $T_{p_1}M_1$ and $T_{\phi_{1}(p_1)}N_1,$ respectively. 
Clearly, $\phi_{1}$ is a conformal Riemannian map with dilation $\lambda_{1} = e^{-2}.$\\
 Further, we consider another pair of Riemannian manifolds $M_2 = R^2-\{(0, 0)\}$ with a Riemannian metric $g_{M_2} = e^{2}du_1^2+e^{2}du_2^2$ and $N_2= R^2$ with Euclidean metric $g_{N_2}$.
 Then, we construct a map $\phi_{2}:(M_2, g_{M_2})\rightarrow (N_2, g_{N_2}),$
 defined as
 $$\phi_{2}(u_1, u_2) = (\sin u_1, \cos u_1).$$
 Clearly, $$ker\phi_{2*} = span\{U_2 = E_2\},$$
$$(ker\phi_{2*})^\perp = span\{Y_2 = E_1\},$$
$$range\phi_{2*} = span\{\phi_{2*}{Y_2}= e^{-2}\cos u_1E_1'-e^{-2}\sin u_1E_2'\},$$
where $\{E_1 = e^{-1}\frac{\partial}{\partial u_1}, E_2= e^{-1}\frac{\partial}{\partial u_2}\}$ and $\{E_1' = \frac{\partial}{\partial v_1}, E_2' = \frac{\partial}{\partial v_2}\}$ are the bases of $T_{p_2}M_2$ and $T_{\phi_{2}(p_2)}N_2,$ respectively.
Then $\phi_{2}$ is a conformal Riemannian map with dilation $\lambda_{2} = e^{-2}.$\\
Now, we construct a map $\phi(=\phi_{1}\times\phi_{2}): M(=M_1\times_f M_2)\rightarrow N(=N_1\times_\rho N_2),$ defined as
 $$\phi(x_1, x_2, u_1, u_2) = \Big(\frac{x_1-x_2}{\sqrt{2}}, 0, \sin u_1, \cos u_1 \Big).$$  
Then, $\phi$ is a conformal Riemannian warped product map with dilation $\lambda = e^{-2}.$

\end{example}

\section{Conclusion}
In order to get a deep understanding of the geometry of warped product spaces, in this paper, we have introduced and studied the notion of Clairaut Riemannian warped product maps and conformal Riemannian warped product maps. The purpose to define Clairaut Riemannian warped product map is to find geodesics on warped product manifolds and discuss their possible applications. Further, we analyse the curvature relations on these maps. Geodesics and curvatures are important factors to understand the geometry of manifolds. We also explored the Bochner type formula in this context, which plays an important role to prove comparison theorems of Ricci curvature and scalar curvature. We examined how Ricci soliton structures behave on warped product manifolds and showed that under certain conditions, the fibers and image of these maps also carry similar geometric structures. In the conformal setting, we discussed conformal isometries, immersions, submersions and showed how changing the metric affects the geometry emphasising the role of O’Neill’s tensor. Additionally, we derived an integral formula for scalar curvature in the conformal case. This study lays the foundation for future work in related areas. Interesting directions include exploring Clairaut conformal warped product maps, statistical versions of these maps, and deeper connections between conformal geometry and warped products.

\section{Acknowledgement}
First author is grateful to the financial support provided by CSIR (Council
of science and industrial research) Delhi, India. File
no.[09/1051(12062)/2021-EMR-I]. The second author is thankful to UGC for providing financial assistance interms of the JRF scholarship vide NTA Ref. No.: 201610070797(CSIR-UGCNET June 2020). The corresponding author is thankful to the
Department of Science and Technology(DST) Government of India for providing
financial assistance in terms of FIST project(TPN-69301) vide the letter
with Ref No.:(SR/FST/MS-1/2021/104).\\

	\noindent J. Yadav, H. Kaur and G. Shanker\newline
Department of Mathematics and Statistics\newline
Central University of Punjab\newline
Bathinda, Punjab-151401, India.\newline
Email: sultaniya1402@gmail.com;\newline gauree.shanker@cup.edu.in;\newline harmandeepkaur1559@gmail.com\newline\\

\end{document}